\documentclass[12pt]{amsart}
\usepackage{amsmath}
\usepackage{amsfonts}
\begin{document}
\numberwithin{equation}{section}

\def\1#1{\overline{#1}}
\def\2#1{\widetilde{#1}}
\def\3#1{\widehat{#1}}
\def\4#1{\mathbb{#1}}
\def\5#1{\frak{#1}}
\def\6#1{{\mathcal{#1}}}

\newcommand{\de}{\partial}
\newcommand{\R}{\mathbb R}
\newcommand{\Ha}{\mathbb H}
\newcommand{\al}{\alpha}
\newcommand{\tr}{\widetilde{\rho}}
\newcommand{\tz}{\widetilde{\zeta}}
\newcommand{\tk}{\widetilde{C}}
\newcommand{\tv}{\widetilde{\varphi}}
\newcommand{\hv}{\hat{\varphi}}
\newcommand{\tu}{\tilde{u}}
\newcommand{\tF}{\tilde{F}}
\newcommand{\debar}{\overline{\de}}
\newcommand{\Z}{\mathbb Z}
\newcommand{\C}{\mathbb C}
\newcommand{\Po}{\mathbb P}
\newcommand{\zbar}{\overline{z}}
\newcommand{\G}{\mathcal{G}}
\newcommand{\So}{\mathcal{S}}
\newcommand{\Ko}{\mathcal{K}}
\newcommand{\U}{\mathcal{U}}
\newcommand{\B}{\mathbb B}
\newcommand{\oB}{\overline{\mathbb B}}
\newcommand{\Cur}{\mathcal D}
\newcommand{\Dis}{\mathcal Dis}
\newcommand{\Levi}{\mathcal L}
\newcommand{\SP}{\mathcal SP}
\newcommand{\Sp}{\mathcal Q}
\newcommand{\A}{\mathcal O^{k+\alpha}(\overline{\mathbb D},\C^n)}
\newcommand{\CA}{\mathcal C^{k+\alpha}(\de{\mathbb D},\C^n)}
\newcommand{\Ma}{\mathcal M}
\newcommand{\Ac}{\mathcal O^{k+\alpha}(\overline{\mathbb D},\C^{n}\times\C^{n-1})}
\newcommand{\Acc}{\mathcal O^{k-1+\alpha}(\overline{\mathbb D},\C)}
\newcommand{\Acr}{\mathcal O^{k+\alpha}(\overline{\mathbb D},\R^{n})}
\newcommand{\Co}{\mathcal C}
\newcommand{\Hol}{{\sf Hol}(\mathbb H, \mathbb C)}
\newcommand{\Aut}{{\sf Aut}(\mathbb D)}
\newcommand{\D}{\mathbb D}
\newcommand{\oD}{\overline{\mathbb D}}
\newcommand{\oX}{\overline{X}}
\newcommand{\loc}{L^1_{\rm{loc}}}
\newcommand{\la}{\langle}
\newcommand{\ra}{\rangle}
\newcommand{\thh}{\tilde{h}}
\newcommand{\N}{\mathbb N}
\newcommand{\kd}{\kappa_D}
\newcommand{\Hr}{\mathbb H}
\newcommand{\ps}{{\sf Psh}}
\newcommand{\Hess}{{\sf Hess}}
\newcommand{\subh}{{\sf subh}}
\newcommand{\harm}{{\sf harm}}
\newcommand{\ph}{{\sf Ph}}
\newcommand{\tl}{\tilde{\lambda}}
\newcommand{\gdot}{\stackrel{\cdot}{g}}
\newcommand{\gddot}{\stackrel{\cdot\cdot}{g}}
\newcommand{\fdot}{\stackrel{\cdot}{f}}
\newcommand{\fddot}{\stackrel{\cdot\cdot}{f}}
\def\v{\varphi}
\def\Re{{\sf Re}\,}
\def\Im{{\sf Im}\,}

\newcommand{\Real}{\mathbb{R}}
\newcommand{\Natural}{\mathbb{N}}
\newcommand{\Int}{\mathbb{Z}}
\newcommand{\UD}{\mathbb{D}}
\newcommand{\clS}{\mathcal{S}}
\newcommand{\gtz}{\ge0}
\newcommand{\gt}{\ge}
\newcommand{\lt}{\le}
\newcommand{\fami}[1]{(#1_{s,t})}
\newcommand{\famc}[1]{(#1_t)}
\newcommand{\ts}{t\gt s\gtz}

\newcommand{\step}[1]{\vskip5mm\begin{itemize}\item[\hbox{[Step #1]}]}
\newcommand{\estep}{\end{itemize}}

\newcommand{\mcite}[1]{\csname b@#1\endcsname}


\def\Label#1{\label{#1}}


\def\cn{{\C^n}}
\def\cnn{{\C^{n'}}}
\def\ocn{\2{\C^n}}
\def\ocnn{\2{\C^{n'}}}
\def\je{{\6J}}
\def\jep{{\6J}_{p,p'}}
\def\th{\tilde{h}}


\def\dist{{\rm dist}}
\def\const{{\rm const}}
\def\rk{{\rm rank\,}}
\def\id{{\sf id}}
\def\aut{{\sf aut}}
\def\Aut{{\sf Aut}}
\def\CR{{\rm CR}}
\def\GL{{\sf GL}}
\def\Re{{\sf Re}\,}
\def\Im{{\sf Im}\,}
\def\U{{\sf U}}

\def\la{\langle}
\def\ra{\rangle}

\emergencystretch15pt \frenchspacing

\newtheorem{theorem}{Theorem}[section]
\newtheorem{lemma}[theorem]{Lemma}
\newtheorem{proposition}[theorem]{Proposition}
\newtheorem{corollary}[theorem]{Corollary}

\theoremstyle{definition}
\newtheorem{definition}[theorem]{Definition}
\newtheorem{example}[theorem]{Example}

\theoremstyle{remark}
\newtheorem{remark}[theorem]{Remark}
\numberwithin{equation}{section}

\title[Loewner chains]{Loewner chains in the unit disk}

\author[M. D. Contreras]{Manuel D. Contreras $^\dag$}

\author[S. D\'{\i}az-Madrigal]{Santiago D\'{\i}az-Madrigal $^\dag$}
\address{Camino de los Descubrimientos, s/n\\
Departamento de Matem\'{a}tica Aplicada II\\
Escuela T\'{e}cnica Superior de Ingenieros\\
Universidad de Sevilla\\
Sevilla, 41092\\
Spain.}\email{contreras@us.es} \email{madrigal@us.es}

\author[P. Gumenyuk]{Pavel Gumenyuk $^\ddag$}
\address{Department of
Mathematics\\ University of Bergen, Johannes Brunsgate 12\\ Bergen 5008, Norway. }
\email{Pavel.Gumenyuk@math.uib.no}

\date{\today }


\thanks{$^\dag$ Partially supported by the \textit{Ministerio
de Ciencia e Innovaci\'on} and the European Union (FEDER), project MTM2006-14449-C02-01, by \textit{La
Consejer\'{\i}a de Educaci\'{o}n y Ciencia de la Junta de Andaluc\'{\i}a} and by the ESF Networking Programme
``Harmonic and Complex Analysis and its Applications''}

\thanks{$^\ddag$ Partially supported by the ESF Networking Programme
``Harmonic and Complex Analysis and its Applications'', the {\it Research Council of Norway} and the {\it
Russian Foundation for Basic Research} (grant \#07-01-00120)}

\begin{abstract}
In this paper we introduce a general version of the notion of
Loewner chains which comes from the new and unified treatment, given
in \cite{BCM1},
of the radial and chordal variant of the Loewner differential equation, which is of special interest in
geometric function theory as well as for various developments it has given rise to, including the famous
Schramm-Loewner evolution.  In this very general setting, we establish a deep correspondence between these
chains and the evolution families introduced in \cite{BCM1}. Among other things, we show that, up to a
Riemann map, such a correspondence is one-to-one. In a similar way as in the classical Loewner theory, we
also prove
that these chains are solutions of a certain partial differential equation which resembles (and includes
as a very particular case) the classical Loewner-Kufarev PDE.

\end{abstract}

\maketitle

\tableofcontents

\section{Introduction}

\subsection{Classical Loewner theory}\label{classical}

In 1923 Loewner~\cite{Loewner} introduced the so-called {\it parametric  method} in geometric function theory, mainly in hope to solve the famous Bieberbach problem about obtaining sharp estimates of Taylor coefficients of normalized holomorphic univalent functions in the unit disk. It is worth recalling that the solution of this problem, given in 1984 by de Branges, relied also on this method. The modern form of the parametric method is mainly due to contributions by Kufarev~\cite{Kufarev} and Pommerenke~\cite{Pommerenke-65}. Let us briefly recall the main constructions (see, e.g.,~\cite[Chapter~6]{Pommerenke}).

Let $f_0(z)=z+a_2z^2+\ldots$ be a holomorphic univalent function in the unit disk $\UD:=\{z:|z|<1\}$. One can always embed this function into a uniparametric family $(f_t)_{t\gtz}$ of holomorphic univalent functions in $\UD$ satisfying the following two properties: $f_t(z)=e^tz+a_2(t)z^2+\ldots$ for any $t\gtz$ and $f_s(\UD)\subset f_t(\UD)$ whenever $t\ge s\gtz$. These type of families are called (classical) Loewner chains. One of the keystones of the parametric method is the fact that every such a family is differentiable in $t$ almost everywhere on $[0,+\infty)$ and independently on $z$. Moreover, they satisfy
the following PDE

\begin{equation}\label{LK_classical}
\frac{\partial f_t(z)}{\partial t}=z \frac{\partial f_t(z)}{\partial z} p(z,t),
\end{equation}
where the driving term $p(z,t)$ is  measurable with respect to $t\in[0,+\infty)$ for all $z\in\UD$ and holomorphic in $z\in\UD$ with $p(0,t)=1$ and $\Re p(z,t)>0$ almost everywhere on $t\in[0,+\infty)$. This equation is called the {\it Loewner\,--\,Kufarev PDE}.

For each $t\gt s\gtz$, the function $\varphi_{s,t}:=f_t^{-1}\circ f_s$ is clearly a holomorphic univalent self-mapping of~$\UD$ and the whole family $(\varphi_{s,t})_{t\gt s\gtz}$ is referred to as the associated {\it evolution family} (sometimes transition family or semigroup family) of the Loewner chain. The remarkable fact is that, fixing $z\in\UD$ and $s\ge0$, the functions $w(t)=\varphi_{s,t}(z)$ are integrals of the characteristic equation for~\eqref{LK_classical}
\begin{equation}\label{LK_classical_ODE}
 \frac{dw}{dt}=-wp(w,t)
\end{equation}
with the initial condition $w(s)=z$. This equation is called the {\it Loewner\,--\,Kufarev ODE} and the right member of the equation, the associated vector field. Note that the family $(\varphi_{s,t})$ is continuous in $t\in[s,+\infty)$ in the compact-open topology of $\mathrm{Hol}(\UD,\mathbb{C})$ for each $s\ge0$, and satisfies the algebraic conditions
\begin{equation}\label{algebraic}
 \varphi_{s,s}=id_\UD,~s\ge0,\text{\mbox{~~~}~and~\mbox{~~~}}\varphi_{s,t}=\varphi_{u,t}\circ\varphi_{s,u},~0\le
s\le u\le t<+\infty.
\end{equation}

Another crucial point in the parametric method is that the function $f_0$ can be reconstructed by means of the integrals of~\eqref{LK_classical_ODE}. Namely,
\begin{equation*}\label{parametric}
 \lim_{t\to+\infty}e^t\varphi_{0,t}=f_0.
\end{equation*}

Equation~\eqref{LK_classical_ODE} can be considered on its own, without any {\it a priori} connection to Loewner chains. However, taking any driving term $p(z,t)$ satisfying the above conditions,
this equation has a unique solution $w(t)=\varphi_{s,t}(z)$, assuming the initial condition~$w(s)=z$. Then, it is possible to define  $f_s:=\lim_{t\to+\infty}e^t\varphi_{s,t}$ and generate in this way a Loewner chain. Clearly, $(\varphi_{s,t})$ is an evolution family associated to this chain $(f_t)$.

In other words, within the framework of the classical parametric
method, there is a one-to-one correspondence between this concept of
evolution families, the driving terms (or the vector fields)
appearing in Loewner equations and the so-called classical Loewner
chains.

\subsection{Chordal Loewner equation}

In his original work~\cite{Loewner}, Loewner paid special attention
to what we call now Loewner chains of slit mappings of~the unit
disk~$\UD$. This Loewner chain $(f_t)$ starts with a conformal
mapping onto the complex plane minus a Jordan curve going to
infinity and, thereafter, the family is
obtained by erasing gradually
this curve. In this case, the ODE equation~\eqref{LK_classical_ODE}
assumes the following form (see, e.\,g., \cite[chapter~III
\S2]{Goluzin} or \cite[chapter~3]{Duren})
\begin{equation}\label{L_ODE}\frac{dw}{dt}=-w\,\frac{\kappa(t)+w}{\kappa(t)-w},\end{equation}
where $\kappa:[0,+\infty)\to\mathbb{R}$ is a continuous function.
The corresponding functions $\varphi_{s,t}=f^{-1}_t\circ f_s$ map
$\UD$ onto $\UD$ with a slit generated by a Jordan curve starting
from the boundary (see \cite[Chapter 17]{Conway2}). These
self-mappings of the unit disk  are normalized at the origin:
$\varphi_{s,t}(0)=0$, $\varphi_{s,t}'(0)>0$.  However, in many
applications, one find quite similar examples but where the natural
normalization is at a boundary point of the unit disk. In this case,
it is possible to consider a real analogue of~\eqref{L_ODE}, the
{\it chordal Loewner equation} (see, e.\,g., \cite[chapter~IV
\S7]{Aleksandrov}), which is traditionally written for the upper
half-plane $\mathbb{U}:=\{z:\mathop{\rm Im} z>0\}$ instead of the
unit disk~$\UD$ because there the associated vector field assumes
the simpler form
\begin{equation}\label{chordal}
\frac{dw}{dt}=\frac{2}{\xi(t)-w},\quad w(0)=z,
\end{equation}
where $\xi:[0,+\infty)\to\mathbb{R}$ is a real-valued driving term. In this chordal context, we could also talk again about driving terms, Loewner chains and evolution families.
In contrast to the chordal variant, classical Loewner theory is mentioned in the recent literature as the {\it radial case}.

The above chordal variant has been extended to cover a wider variety of new situations. For instance, the relationship between some kind of what we can name chordal Loewner chains and what deserve be named chordal evolution families has been considered by Goryainov and Ba in~\cite{Goryainov-Ba} and by Bauer in~\cite{Bauer}.

A recent burst of interest in Loewner theory is due in part to the
so-called {\it Schramm\,--\,Loewner evolution} (SLE, also known as
{\it stochastic Loewner evolution}), introduced in 2000 by
Schramm~\cite{Schramm}. SLE is an evolution model similar to Loewner
chains (namely, given by equation~\eqref{L_ODE} or~\eqref{chordal})
but with the driving term defined via a Brownian motion. In other
words, it is a probabilistic version of the previously known radial
and chordal Loewner chains. Both radial and chordal cases of SLE
have important applications. In fact, they turn out to be very
useful tools for the study of conformally invariant scaling limits
of some classical statistical 2-dimensional lattice models,
see~[\mcite{Lawler}\,--\,\mcite{Lawler-Schramm-Werner}].

Some recent developments concerning the relationship between properties of
the driving term and the geometry of solutions to~\eqref{L_ODE} and~\eqref{chordal} can be found in~\cite{Lind, Marshall-Rohde, Prokhorov-Vasiliev}.

\subsection{Generalization of classical evolution families}

Loewner\,--\,Kufarev ODE \eqref{LK_classical_ODE}  {\it defines a holomorphic evolution} in the unit disk. That is, for any initial point $z\in\UD$ and any starting instance $s\ge0$, the solution $w=w(t)$ to the initial value problem~$w(s)=z$ for  equation~\eqref{LK_classical_ODE} is unique, exists for all $t\ge s$, and the dependence of $w(t)$ on $z$ reveals a holomorphic self-mapping of~$\UD$. The same is true for the chordal Loewner equation~\eqref{chordal} and its generalizations when being rewritten for $\UD$. Some natural questions arise: {\it are there other examples of ODE with the same property?, what is the most general form of such type of equations?, is it possible to unify these holomorphic evolutions, bearing in mind the many similarities between them?}

The answer for the autonomous case (the vector field is of the form $dw/dt=G(w)$) comes from the theory of one-parametric semigroups of holomorphic functions (see the definition in Section~\ref{semi}). They have important applications in the theory of operators acting on spaces of analytic functions (see, e.g.,~\cite{Shapiro,Shoikhet}) as well as in the theory of stochastic processes (see, e.g.,~\cite{Goryainov1, Goryainov2}).  Berkson and Porta~\cite{Berkson-Porta} found the most general form of such a function~$G$, namely
\[
G(z)=(\tau-z)(1-\overline{\tau}z)p(z),\text{ \quad}z\in\mathbb{D},
\]
where $p$ is a holomorphic function in $\D$ with $\Re p(z)\geq0$ and $\tau\in\oD$ (again see Section~\ref{semi} for more  details).

However, in the non-autonomous case and as far as we know, there
were no satisfactory answers to the above questions before
~\cite{BCM1}. Certainly, a large number of examples related to
chordal and radial Loewner differential equations has been treated
in the literature but, at the same time, one can also find several
(similar but different) notions
playing the role of Loewner
chains, vector fields or, specially, evolution families.
For instance, in~\cite{Goryainov} some classes of holomorphic univalent self-mappings, closed with respect to composition, are considered and evolution families within these classes are defined as
two-parametric families $(\varphi_{s,t})_{0\leq s\leq t}$ continuous
with respect to $t\in[s,+\infty)$ in the open-compact topology of
$\mathrm{Hol}(\UD,\mathbb{C})$ for each $s\ge0$ and satisfying the
algebraic conditions~\eqref{algebraic}. Moreover, in order to
describe evolution families by means of differential equations, an
additional condition is also imposed: namely, a certain functional
applied to $\varphi_{0,t}$ is required to be (locally) absolutely
continuous with respect to~~$t$. It is worth comparing this approach
with the very classical case,
where one can regard the equality $\varphi_{0,t}'(0)=e^{-t}$ as a kind of additional condition ensuring differentiability in $t$.

As we have just partially said, answers for the above questions under very general assumptions follow from results of the recent paper~\cite{BCM1}  by Bracci and the first two authors of this paper. Taking the whole class of holomorphic self-maps of $\UD$, they introduced a general notion of {\it evolution family in the unit disk} which includes, as very particular cases, one-parametric semigroups as well as all of those evolution families arising in Loewner theory, both for the radial and chordal variants. Now we cite their definition. Note that the functions $\varphi_{s,t}$ are {\it not assumed a priori to be univalent} in $\UD$.

\begin{definition}\label{def-ev}
A family $(\varphi_{s,t})_{0\leq s\leq t<+\infty}$ of holomorphic self-maps of the unit disc  is an {\sl
evolution family of order $d$} with $d\in [1,+\infty]$ (in short, an {\sl $L^d$-evolution family}) if

\begin{enumerate}
\item[EF1.] $\varphi_{s,s}=id_{\mathbb{D}},$

\item[EF2.] $\varphi_{s,t}=\varphi_{u,t}\circ\varphi_{s,u}$ for all $0\leq
s\leq u\leq t<+\infty,$

\item[EF3.] for all $z\in\mathbb{D}$ and for all $T>0$ there exists a
non-negative function $k_{z,T}\in L^{d}([0,T],\mathbb{R})$ such that
\[
|\varphi_{s,u}(z)-\varphi_{s,t}(z)|\leq\int_{u}^{t}k_{z,T}(\xi)d\xi
\]
for all $0\leq s\leq u\leq t\leq T.$
\end{enumerate}
\end{definition}

One of the main results of~\cite{BCM1} is that any evolution family $(\varphi_{s,t})$ can be obtained via solutions to an ODE of the form $dw/dt=G(w,t)$. Moreover, they characterize all the functions (or, in other words, all the vector fields) $G$ that generate evolution families. Indeed, these vector fields resemble to a non-autonomous (the variable $t$ is present) version of the celebrated Berkson-Porta representation theorem (see Section~\ref{evolution-families} for further definitions and full statements of these results). Nevertheless, a one-to-one correspondence between evolution families and certain type of vector fields is established in that paper. There, it is also explained how to recover the semigroup, radial and chordal cases in this new framework. Indeed, the three authors were able to formulate a similar theory of generalized evolution families for arbitrary hyperbolic complex manifolds \cite{BCM2}.

In ~\cite{BCM1}, the following natural question was left opened
implicitly: is there a generalized notion of Loewner chain which can
be put in one-to-one correspondence with those generalized evolution
families or, equivalently, with those generalized Berkson-Porta
vector fields? In the next subsection, we deal with this question
presenting our main results about it.

\subsection{Main results}

As we mentioned in Section~\ref{classical}, Loewner\,--\,Kufarev equation~\eqref{LK_classical_ODE} generates a special type of evolution families and there is a one-to-one correspondence between such evolution families and classical Loewner chains.

In this paper we consider the analogous question for arbitrary evolution families in the sense of Definition~\ref{def-ev}. First of all, we give a suitable definition of Loewner chain for our general setting.

\begin{definition}\label{def-lc}
A family $(f_t)_{0\leq t<+\infty}$ of holomorphic maps of the unit disc will be called a {\sl Loewner chain of order $d$}
with $d\in [1,+\infty]$ (in short, an {\sl $L^d$-Loewner chain}) if

\begin{enumerate}
\item[LC1.] each function $f_t:\D\to\C$ is univalent,

\item[LC2.] $f_s(\D)\subset f_t(\D)$ for all $0\leq s < t<+\infty,$

\item[LC3.] for any compact set $K\subset\mathbb{D}$ and  all $T>0$ there exists a
non-negative function $k_{K,T}\in L^{d}([0,T],\mathbb{R})$ such that
\[
|f_s(z)-f_t(z)|\leq\int_{s}^{t}k_{K,T}(\xi)d\xi
\]
for all $z\in K$ and all $0\leq s\leq t\leq T$.
\end{enumerate}
A  Loewner chain $\famc f$ will be said to be {\it normalized} if  $f_0(0)=0$ and $f'_0(0)=1$ (notice that we only
normalize the function $f_0$).
\end{definition}

Our main results concerning relations between Loewner chains and evolution families are stated in the following three theorems.

\begin{theorem}\label{T1}
For any Loewner chain $(f_t)$ of order $d\in[1,+\infty]$, if we define
$$
\varphi_{s,t}:= f_t^{-1}\circ f_s, \quad 0\le s\le t,
$$
then $(\v_{s,t})$ is an evolution family of the same order~$d$. Conversely, for any evolution family $(\v_{s,t})$ of order~$d\in[1,+\infty]$, there exists a Loewner chain $(f_t)$ of the same order~$d$ such that the following equation holds
\begin{equation}\label{main_EV_LC}
 f_t\circ\varphi_{s,t}=f_s,\quad 0\le s\le t.
\end{equation}
\end{theorem}

\begin{definition}
A Loewner chain $\famc f$ is said to be {\it associated with} an evolution family~$\fami\varphi$ if
it satisfies~\eqref{main_EV_LC}.
\end{definition}

\begin{remark}
 We will actually prove (see Lemma~\ref{equationimpliesLC}) that  {any Loewner chain $(f_t)$ associated with an evolution family $(\v_{s,t})$ of order~$d\in[1,+\infty]$  must be of the same order~$d$.}
\end{remark}

In general, fixed the evolution family $(\v_{s,t})$, the algebraic equation~\eqref{main_EV_LC} does not defined a unique Loewner chain. In fact, in some case, a plenty of different Loewner chains are associated with the same evolution family. The following theorem gives necessary and sufficient conditions for the uniqueness for a normalized Loewner chain associated with a given evolution family.

\begin{theorem}\label{T2}
 Let $(\v_{s,t})$ be an evolution family. Then there exists a unique
 normalized Loewner chain $(f_t)$ associated with~$(\v_{s,t})$ such that $\cup_{t\ge0}f_t(\UD)$ is either an Euclidean disk or the whole complex plane~$\C$. Moreover, the following statements are equivalent:
\begin{itemize}
\item[(i)] the family $(f_t)$ is the only normalized Loewner chain associated with the evolution family~$(\v_{s,t})$;
\item[(ii)] for all $z\in\UD$, $$\beta(z):=\lim_{t\to+\infty}\frac{|\varphi_{0,t}'(z)|}{1-|\varphi_{0,t}(z)|^2}=0;$$
\item[(iii)] there exist at least one point $z\in\UD$ such that $\beta(z)=0$;
\item[(iv)] $\bigcup\limits_{t\gtz}f_t(\UD)=\C$.
\end{itemize}
\end{theorem}
The Loewner chain $(f_t)$ in the above theorem will be called the {\it standard Loewner chain} associated with the evolution family~$(\varphi_{s,t})$.

In case of non-uniqueness (when conditions (i)\,--\,(iv) in Theorem~\ref{T2} fail to be satisfied), we provide an explicit formula expressing all the associated normalized Loewner chains by means of the standard Loewner chain plus some Riemman map.
 In some sense, this formula tell us that the evolution procedures described by our Loewner chains are essentially unique up to a choice of the simply connected domain they are located in.
Denote by  $\clS$ the class of all univalent holomorphic functions $h$ in the unit disk~$\UD$,
normalized by $h(0)=h'(0)-1=0$.

\begin{theorem}\label{T3}
Suppose that under conditions of Theorem~\ref{T2},
$$\Omega:=\bigcup\limits_{t\gtz}f_t(\UD)\neq\C.$$ Then $\Omega=\{z:|z|<1/\beta(0)\}$ and the set $\mathcal L[(\v_{s,t})]$ of all normalized Loewner chains $(g_t)$ associated with the evolution family~$(\v_{s,t})$, is given by the formula
$$
\mathcal L[(\v_{s,t})]=\big\{(g_t)_{t\ge0}:g_t(z)=h\big(\beta(0) f_t(z)\big)/\beta(0),~ h\in\clS\big\}.
$$
\end{theorem}

In Section~\ref{evolution-families} we state some results from~\cite{BCM1} along with the necessary definitions. Moreover, we prove new statements concerning evolution families (see Definition~\ref{def-ev}), which we later use to obtain the main results of the paper.

In Section~\ref{main-section} we reformulate and prove the theorems
stated above. Namely, Theorem~\ref{T1} follows from
Theorems~\ref{LCimpliesEF} and \ref{Gum_exist}, while
Theorems~\ref{T2} and \ref{T3} follow from Theorem~\ref{Gum_uniq}
and Proposition~\ref{beta_prop}. Besides that, and in some cases, we
establish a necessary and sufficient
condition~(Theorem~\ref{no_univ_assum}) for a uniparametric family $(f_t)_{t\gtz}$ of holomorphic  ({\it but not a priori univalent}) maps defined in $\UD$ to be a normalized Loewner chain associated with a given evolution family.

In Section~\ref{diff_eq_sect} we find an analogue (Theorem~\ref{diff_eq_t}) of the Loewner\,--\,Kufarev PDE in this abstract context. We also show that there is a one-to-one correspondence between our concept of generalized Loewner chain and the generalized Berkson-Porta vector fields shown in \cite{BCM1}.

In Section~\ref{semi} we consider the special case of evolution families induced by semigroups of holomorphic functions in~$\UD$. In  particular, we show that the uniqueness of the K\oe nigs function is a consequence of Theorems~\ref{T1} and \ref{T2}.

\section{Evolution families and Herglotz vector fields in the unit disk}
\label{evolution-families}

Here we collect some known and new statements on evolution families (see Definition~\ref{def-ev}).

Let us first of all note that by \cite[Corollary 6.3]{BCM1}, {\it given an evolution family $(\varphi_{s,t})$, every function $\varphi_{s,t}$ is
univalent.} The following statement turns out to be also quite useful.
\begin{lemma}\cite[Lemma 3.6]{BCM1}
\label{EF-boundedness} Let $(\varphi_{s,t})$ be an evolution family in the unit disc
$\mathbb{D}$ of order $d\in[1,+\infty]$.  Then for each $0<T<+\infty$ and $0<r<1,$ there exists $R=R(r,T)<1$ such that
\[
|\varphi_{s,t}(z)|\leq R
\]
for all $0\leq s\leq t\leq T$ and $|z|\leq r.$
\end{lemma}

Any evolution family~$(\varphi_{s,t})$ is differentiable almost everywhere with respect to $t$. Besides the proof of this fact,  a characterization of all vector fields generating evolution families in the disk is established in~\cite{BCM1}. In order to give a strict statement of this result we need the following
\begin{definition}
\label{Definicion-VF} Let $d\in [1,+\infty]$. A {\sl weak holomorphic vector field of order $d$} in the unit
disc $\mathbb{D}$ is a function $G:\mathbb{D}\times\lbrack0,+\infty)\rightarrow \mathbb{C}$ with the following
properties:

\begin{enumerate}
\item[WHVF1.] For all $z\in\mathbb{D},$ the function $\lbrack
0,+\infty)\ni t\mapsto G(z,t)$ is measurable;

\item[WHVF2.] For all $t\in\lbrack0,+\infty),$ the function $
\mathbb{D}\ni z\mapsto G(z,t)$ is holomorphic;

\item[WHVF3.] For any compact set $K\subset\mathbb{D}$ and  all $T>0$ there
exists a non-negative function $k_{K,T}\in L^{d}([0,T],\mathbb{R})$ such that
\[
|G(z,t)|\leq k_{K,T}(t)
\]
for all $z\in K$ and for almost every $t\in\lbrack0,T].$
\end{enumerate}
Moreover, we say that $G$ is a {\sl (generalized) Herglotz vector
field} (of order $d$) if for almost every $t\in [0,+\infty)$ it
follows $G(\cdot, t)$ is the infinitesimal generator of a semigroup
of holomorphic functions (see Section~\ref{semi} for further details
about semigroups of analytic functions and their infinitesimal
generators).
\end{definition}

\begin{theorem}\cite[Theorems 6.2, 4.8]{BCM1}
\label{EF-bijection-VF} For any evolution family $(\v_{s,t})$ of order $d\in[1,+\infty]$ there exists a
(essentially) unique  Herglotz vector field $G(z,t)$ of order $d$ such that for all~$z\in \D$,
\begin{equation}\label{main-eq}
\frac{\de \v_{s,t}(z)}{\de t}=G(\v_{s,t}(z),t), \quad \hbox{a.e. $t\in [0,+\infty)$}.
\end{equation}
Conversely, for any Herglotz vector field $G(z,t)$ of order $d\in[1,+\infty]$
there exists a unique evolution family $(\v_{s,t})$ of order $d$  such that \eqref{main-eq} is satisfied.
\end{theorem}
Here by {\it essential uniqueness} we mean that two Herglotz vector fields $G_1(z,t)$ and $G_2(z,t)$ corresponding to the same evolution family must coincide for a.e. $t\gtz$.

Herglotz vector fields can be further characterized in similar terms of the Berkson\,--\,Porta representation of infinitesimal generators.

\begin{definition}\label{def-Her-fun}
Let $d\in [1,+\infty]$. A {\sl Herglotz function of order $d$} is a function $p:\mathbb{D}\times\lbrack0,+\infty
)\to\mathbb{C}$ with the following properties:

\begin{enumerate}
\item[HF1.] For all $z\in\mathbb{D},$ the function $\lbrack0,+\infty
)\ni t \mapsto p(z,t)\in\mathbb{C}$ belongs to $L_{loc}^{d}([0,+\infty),\mathbb{C})$;

\item[HF2.] For all $t\in\lbrack0,+\infty),$ the function
$\mathbb{D}\ni z \mapsto p(z,t)\in\mathbb{C}$ is holomorphic;

\item[HF3.] For all $z\in\mathbb{D}$ and for all $t\in\lbrack0,+\infty),$ we
have $\Re p(z,t)\geq0.$
\end{enumerate}
\end{definition}

\begin{theorem}\cite[Theorems 1.2]{BCM1}\label{Berkson-Porta-for-Herglotz}
Let $G(z,t)$ be a Herglotz vector field  of order $d\in[1,+\infty]$ in the unit disc. Then there exist a (essentially)
unique measurable function  $\tau:[0,+\infty)\to \oD$  and a Herglotz function $p(z,t)$ of order $d$ such that
for all $z\in \D$
\begin{equation}\label{Herglotz-vf-main}
G(z,t)=(z-\tau(t))(\overline{\tau(t)}z-1)p(z,t),\quad \hbox{a.e. $t\in [0,+\infty)$}.
\end{equation}
Conversely, given a  measurable function  $\tau:[0,+\infty)\to \oD$  and a Herglotz function $p(z,t)$ of order
$d\in[1,+\infty]$, equation \eqref{Herglotz-vf-main} defines a Herglotz vector field of order $d$.
\end{theorem}

There is thus an (essentially) one-to-one correspondence between evolution families $(\v_{s,t})$ of order $d\in[1,+\infty]$, Herglotz vector fields $G(z,t)$ of order $d$, and couples $(p,\tau)$ of Herglotz functions $p(z,t)$ of
order $d$ and measurable functions $\tau:[0,+\infty)\to \oD$. In what follows we say that the couple $(p,\tau)$
is the {\sl Berkson\,--\,Porta data} for $(\v_{s,t})$.

Now we state and prove some new assertions concerning evolution families, which we use in the proof of the main results.

Denote by $AC^d(X,Y)$, $X\subset \R$, $d\in[1,+\infty]$, the class of all locally absolutely continuous
functions $f:X\to Y$ such that the derivative $f'$ belongs to $L_{\rm loc}^d(X)$.
\begin{proposition}
\label{EF-properties} Let $(\varphi_{s,t})$ be an evolution family of order $d\in[1,+\infty]$. Then the following statements hold:
\begin{enumerate}
\item For any compact set $K\subset\mathbb{D}$ and all $T>0$ there exists a
non-negative function $k_{K,T}\in L^{d}([0,T],\mathbb{R})$ such that
\[
|\varphi_{s,u}(z)-\varphi_{s,t}(z)|\leq\int_{u}^{t}k_{K,T}(\xi)d\xi
\]
for all $0\leq s\leq u\leq t\leq T$ and all $z\in K$.
\item For every $z\in \mathbb{D}$ the maps  $a(t):=\v_{0,t}(z)$ and
$b(t):=\v'_{0,t}(z)$ belong to $AC^d([0,+\infty),\C )$ and $b(t)\neq 0$ for all  $t\in [0,+\infty)$.
\end{enumerate}
\end{proposition}
\begin{proof} By Theorem \ref{EF-bijection-VF}, there is a Herglotz vector
field of order $d$ such that for all $z\in\D$
\[
\frac{\partial \v_{s,t}(z)}{\partial t}=G(\v_{s,t}(z),t),\quad \ \text{a.e. } t\in[0,+\infty).
\]
{\it Proof of (1).} By the very definition of Herglotz vector field there exists a non-negative function $k_{K,T}\in
L^{d}([0,T],\R)$ such that
\begin{equation}\label{VF-assocated-with-EF}
    |G(z,t)|\leq k_{K,T}(t)
\end{equation}
for all $z\in K$ and for almost every $t\in\lbrack0,T].$ Therefore, statement~(1) is an easily consequence of  the
following inequalities
\[ |\varphi_{s,u}(z)-\varphi_{s,t}(z)| =
  \left|\int_u^t\frac{\partial \v_{s,\xi}(z)}{\partial \xi}d\xi\right|
   = \left|\int_u^tG(\v_{s,\xi}(z),\xi)d\xi\right|
   \leq  \int_u^tk_{K,T}(\xi)d\xi .
\]
{\it Proof of (2).} From the very definition of Herglotz vector field, evolution family of order $d$, and inequality~\eqref{VF-assocated-with-EF} it follows that the map $a$ belongs to $AC^d([0,+\infty),\C )$. Moreover, since the
functions $\v_{s,t}$ are univalent \cite[Corollary 6.3]{BCM1}, we have $b(t)\neq 0$ for all $t$. Fix $T\in(0,+\infty)$
and $z\in\D$. There is $R<1$ such that $|\varphi_{0,t}(z)| < R$ for all $t\in[0,T]$. Then there is $k_{R,T}\in
L^{d}([0,T],\R)$ such that
\begin{equation*}
    |G(w,t)|\leq k_{R,T}(t)
\end{equation*}
for all $|w|\leq R$ and for almost every $t\in\lbrack0,T].$ Therefore,
\begin{eqnarray*}
|  b'(t) |&=& \left|\frac{1}{2\pi}\frac{\partial}{\partial t}\left(
\int_{C(0,R)^+}\frac{\v_{0,t}(w)}{w^2}dw\right)\right|
=\left|\frac{1}{2\pi} \int_{C(0,R)^+}\frac{\partial}{\partial t}\left(\frac{\v_{0,t}(w)}{w^2}\right)dw\right|\\
   &=&
\left|\frac{1}{2\pi}\left( \int_{C(0,R)^+}G(\v_{0,t}(w),t)dw\right)\right|
   \leq  \frac{1}{R} k_{R,T}(t)
\end{eqnarray*}
for almost every $t\in\lbrack0,T]$, where $C(0,R)^+$ stands for the positively oriented circle of radius~$R$ centered at the point~$z=0$.  This implies that $b$ belongs to $AC^d([0,+\infty),\C )$ and therefore completes the proof.
\end{proof}

It appears to be useful to consider evolution families that consists of automorphisms of~$\UD$. The following example is the most general form of such evolution families.
\begin{example}
\label{EF-automorphims} Take two functions $a\in AC^d([0,+\infty),\D)$ and $b\in AC^d([0,+\infty),\partial\D)$
and write
$$
h_t(z):=\frac{b(t)z+a(t)}{1+b(t)\overline{ a(t)}z} \ \text {for all\ } t\geq0 \ \text {and all } z\in\D.
$$
Then $(h_t\circ h_s^{-1})$ and $(h_t^{-1}\circ h_s)$ are evolution families of order $d$. Indeed, it is clear
that both families of functions satisfy EF1 and EF2. Moreover, for any $T<+\infty$ and $z\in\D$  there exists $R<1$ such that
\[
|h_s^{-1}(z)|=\left|\overline{b(s)}\frac{z-a(s)}{1-\overline{a(s)}z}\right|\leq R,\quad 0\le s\le T.
\]
Denote $w=h^{-1}_s(z)$. Then we have
\begin{eqnarray*}
  |h_t\circ h_s^{-1}(z)-h_u\circ h_s^{-1}(z)| &=&  |h_t(w)-h_u(w)|
   = \left|\frac{b(t)w+a(t)}{1+b(t)\overline{ a(t)}w}-
   \frac{b(u)w+a(u)}{1+b(u)\overline{ a(u)}w}\right| \\
   &\leq& \frac{2}{(1-R)^2} \left( |b(t)-b(u)|+|a(t)-a(u)|\right)
\end{eqnarray*}
for all $0\leq s\leq u\leq t\leq T$. These inequalities and the hypothesis on $a$ and $b$ imply that the
family $(h_t\circ h_s^{-1})$ satisfies EF3. Similarly, the family $(h_t^{-1}\circ h_s)$ satisfies EF3 as well.
\end{example}

The following lemma allows us to transform evolution families by means of  time-dependent changes of variable in the unit disk.
\begin{lemma}
\label{EF-lemma-decomposition} Let $(\psi_{s,t})$ be an evolution family or order $d\in[1,+\infty]$ and take two functions
$a\in AC^d([0,+\infty),\D)$ and $b\in AC^d([0,+\infty),\partial\D)$. Write $\v_{s,t}=h_t\circ\psi_{s,t}\circ
h_s^{-1}$ and $\tilde{\v}_{s,t}=h_t^{-1}\circ\psi_{s,t} \circ h_s$, where
$$
h_t(z):=\frac{b(t)z+a(t)}{1+b(t)\overline{a(t)}z} \ \text {for all\ } t\geq0 \ \text {and all\ } z\in\D.
$$
Then $(\v_{s,t})$ and $(\tilde{\v}_{s,t})$ are evolution families of order $d$.
\end{lemma}
\begin{proof} We present the proof for the family $(\v_{s,t})$ and leave to
the reader the one for the family $(\tilde{\v}_{s,t})$ which is quite similar.

It is clear that the functions $(\v_{s,t})$ satisfy properties EF1 and EF2. So we just have to prove that this
family of functions satisfy EF3.

Notice that, by Example \ref{EF-automorphims}, $(h_t\circ h_s^{-1})$ is an evolution family. Fix $z\in\D$ and
$T\in(0,\infty)$. By Lemma \ref{EF-boundedness} and the continuity of the functions $a$ and $b$, there exists a
number $R<1$ such that
\[
|\psi_{s,t}\circ h_s^{-1}(z)|\leq R \ \text{ and \ } |\v_{s,t}(z)|=|h_t\circ\psi_{s,t}\circ h_s^{-1}(z)|\leq R
\]
for all $0\leq s\leq t \leq T$. Therefore, by Proposition \ref{EF-properties} applied to the evolution families
$(h_t\circ h_s^{-1})$ and $(\psi_{s,t})$, there are two functions $k_1, k_2\in L^{d}([0,T],\mathbb{R})$ such
that
\begin{equation}\label{*}
|\psi_{s,u}(w)-\psi_{s,t}(w)|\leq\int_{u}^{t}k_1(\xi)d\xi \ \text {and \ } |h_u\circ h_s^{-1}(w)-h_t\circ
h_s^{-1}(w)|\leq\int_{u}^{t}k_2(\xi)d\xi
\end{equation}
for all $0\leq s\leq u\leq t\leq T$ and whenever $|w|\leq R$. Moreover, there is a positive number $M$ such that
\begin{equation}\label{**}
|h_t(w_1)-h_t(w_2)|\leq M|w_1-w_2|
\end{equation}
whenever $t\in[0,T]$ and $|w_1|,|w_2|\leq R$. Now, let us fix $0\leq s\leq u\leq t\leq T$ and write
$z_1=\psi_{s,u}(h_s^{-1}(z))$ and $z_2=h_u(z_1)$. Note that $|z_1|,|z_2|\leq R$. The following chain of
inequalities (where we use \eqref{*} and \eqref{**}) allows us to complete the proof
\begin{eqnarray*}
|\v_{s,t}(z)-\v_{s,u}(z)| &=&|\v_{u,t}(\v_{s,u}(z))-\v_{s,u}(z)|
=|h_t\circ\psi_{u,t}(z_1)-h_u(z_1)| \\
&\leq & |h_t\circ\psi_{u,t}(z_1)-h_t(z_1)|+|h_t(z_1)-h_u(z_1)|\\
&\leq & M|\psi_{u,t}(z_1)-z_1|+|h_t(z_1)-h_u(z_1)|\\
&= & M|\psi_{u,t}(z_1)-\psi_{u,u}(z_1)|+|h_t\circ h_u^{-1}(z_2)-h_u\circ h_u^{-1}(z_2)|\\
&\leq& \int_{u}^{t}(Mk_1(\xi)+k_2(\xi))d\xi.
\end{eqnarray*}
\end{proof}

Now we use Lemma~\ref{EF-lemma-decomposition} in order to establish a kind of  decomposition for a given evolution family.
\begin{proposition}
\label{EF-decomposition} Let $(\v_{s,t})$ be an evolution family of
order $d\in[1,+\infty]$. Then there exist unique $a\in
AC^d([0,+\infty),\D)$, $b\in AC^d([0,+\infty),\partial\D)$,
 and $\psi_{s,t}:\D\to\D$, $0\leq s\leq t<+\infty$, such that the following assertions hold
\begin{enumerate}
  \item $a(0)=0$, $b (0)=1$,
  \item $(\psi_{s,t})$ is an evolution family of order $d$ such that $\psi_{s,t}(0)=0$
  and $\psi'_{s,t}(0)>0$ for all $0\leq s\leq t$,
  \item $\v_{s,t}=h_t\circ\psi_{s,t}\circ h_s^{-1}$ for all $0\leq s\leq t<+\infty$, where
$$
h_t(z):=\frac{b(t)z+a(t)}{1+b(t)\overline{a(t)}z}, \quad t\geq0,~~ z\in\D.
$$
\end{enumerate}
\end{proposition}
\begin{proof} Write $a (t)= \v_{0,t}(0)$ and
$b (t)=\frac{\v'_{0,t}(0)}{|\v'_{0,t}(0)|}$. By Proposition
\ref{EF-properties}, $a\in AC^d([0,+\infty),\D)$ and $b\in
AC^d([0,+\infty),\partial\D)$. Now define $h_t$ as in the statement
of the proposition and take $\psi_{s,t}=h_t^{-1}\circ\v_{s,t}\circ
h_s$. Notice that $h_0$ is the identity, $h_t(0)=a(t)$ and
$h'_t(0)=b(t)(1-|a (t)|^2)$. By Lemma \ref{EF-lemma-decomposition},
the family $(\psi_{s,t})$ is an evolution family of order $d$.
Moreover, from the very definition of $a$ it follows that
$\psi_{0,t}(0)=0$ for all $t$. Using EF2, we deduce that
$\psi_{s,t}(0)=0$ for all $s\leq t$. In a similar way, we show that
$\psi'_{0,t}(0)=\frac{|\v'_{0,t}(0)|}{1-|a(t)|^2}>0$ for all $t$ and
then $\psi'_{s,t}(0)>0$ for all $0\leq s\leq t$.

The uniqueness is clear because from the equality $\v_{s,t}=h_t\circ\psi_{s,t}\circ h_s^{-1}$ we deduce that $a
(t)= h_t(0)=h_t(\psi_{0,t}(0))=\v_{0,t}(0)$, $b (t)=\frac{\v'_{0,t}(0)}{|\v'_{0,t}(0)|}$ (which defines the functions
$h_t$ uniquely) and $\psi_{s,t}=h_t^{-1}\circ\v_{s,t}\circ h_s$. The proof is now complete.
\end{proof}

The following result gives the converse of Proposition \ref{EF-properties}(2).
\begin{proposition}\label{EF4_prop}
 Let $(\varphi_{s,t})$ be a family of holomorphic self-maps of~$\D$.
Suppose that conditions EF1 and EF2 are fulfilled. Then condition EF3 is equivalent to the following condition:
\begin{itemize}
\item[EF4.]  The maps  $a(t):=\v_{0,t}(0)$ and
$b(t):=\v'_{0,t}(0)$ belong to $AC^d([0,+\infty),\C )$ and $b(t)\neq 0$ for all  $t\in [0,+\infty)$.
\end{itemize}
\end{proposition}
\begin{proof} By Proposition \ref{EF-properties} any evolution family satisfies EF4.

Let  $(\varphi_{s,t})$ be a family of holomorphic self-maps of the unit disk satisfying EF1, EF2, and EF4. Write
$$
h_t(z):=\frac{b_0(t)z+a(t)}{1+b_0(t)\overline{a(t)}z}\quad \ \text {for all \ } t\geq0 \ \text {and all \ } z\in\D,
$$
where $b_0(t)=b(t)/|b(t)|$. Define $\psi_{s,t}=h_t^{-1}\circ\v_{s,t}\circ h_s$ for all $0\leq s\leq t<+\infty$. It is clear that the
family $(\psi_{s,t})$ satisfies EF1, EF2, $\psi_{s,t}(0)=0$, and $\psi_{0,t}'(0)=|b(t)|/(1-|a(t)|^2)$ for all
$0\leq s\leq t$. Using \cite[Theorem 7.3]{BCM1} with $\tau=z_0=0$ in that statement, we deduce that
$(\psi_{s,t})$ is an evolution family of order $d$. Finally, we just have to apply Lemma
\ref{EF-lemma-decomposition} to deduce that $(\varphi_{s,t})$ is also an evolution family of order $d$.
\end{proof}

\section{Loewner chains and evolution families}
\label{main-section}

In this section we reformulate and prove our main results connecting evolution families with Loewner chains in a way
similar to the one given in classical Loewner theory.

First of all we prove that any Loewner chain of order $d\in[0,+\infty]$ generates an evolution family of the same order.

\begin{theorem}
\label{LCimpliesEF}
Let $(f_{t})$ be a Loewner chain of order $d\in[1,+\infty]$. Set%
\[
\varphi_{s,t}(z):=f_{t}^{-1}(f_{s}(z)),\text{ }z\in\mathbb{D},\text{ }0\leq s\leq t.
\]
Then $(\varphi_{s,t})$ is a well-defined evolution family of order $d$ in the unit disk and (trivially)
satisfies the equality
\[
f_{t}(\varphi_{s,t}(z))=f_{s}(z),\text{ }z\in\mathbb{D},\text{ }0\leq s\leq t.
\]
\end{theorem}

\begin{proof}
The proof of this theorem is quite long so we have divided it into several steps
of independent interest on their own. In what follows, $\Omega _{t}:=f_{t}(\mathbb{D}),$ $t\geq 0.$ We also
comment that $\mathop{\rm ins}\Gamma $ will denote the interior of a Jordan curve
$\Gamma $ and $N(g,\Gamma )$ stands for the number of zeros (counting
multiplicity), inside a rectifiable Jordan curve $\Gamma $ contained in $\mathbb{D}$, of a holomorphic map $g$ defined in the whole unit disk. Finally, by $\mathop{\rm ind}(\Gamma,\xi)$ we denote the index of a closed rectifiable curve~$\Gamma$ with respect to a point~$\xi$, and  $D(\xi,r):=\{z\in\mathbb{C}:|z-\xi|<r\}$.

\step1 \textit{For every }$t\geq 0\,$\textit{\ and every }$\omega \in
\Omega _{t},$\textit{\ there exist }$\varepsilon>0,$\textit{\ }$\delta >0$\textit{\
and a rectifiable Jordan curve }$\gamma $\textit{\ with }$\gamma \cup
\mathop{\rm ins}\gamma \subset \mathbb{D}$\textit{\ such that the following
``locally uniform formula for the inverses" holds:}%
\begin{equation*}
f_{u}^{-1}(w)=\frac{1}{2\pi i}\int_{\gamma }\frac{\xi f_{u}^{\prime }(\xi )}{%
f_{u}(\xi )-w}d\xi,
\end{equation*}
\textit{whenever }$u\in \lbrack t-\delta ,t+\delta ]\cap
\lbrack 0,+\infty )$\textit{ and } $w\in D(\omega,\varepsilon)$.

\estep

Fix $t\geq 0$ and $\omega \in \Omega _{t}.$ Denote $z_{0}:=f_{t}^{-1}(\omega
)\in \mathbb{D}$ and choose any $r\in\big(|z_0|,1\big)$ and $R\in(r,1)$. Consider
the complex domain $D_{t}:=f_{t}(D(0,r))\subset \Omega _{t}$ and define
$\gamma $ as the positively oriented circle of radius $R$  centered at the origin.
Since $f_{t}$ is univalent, it follows from the Argument Principle that for
each $w\in D_{t},$%
\begin{equation*}
\frac{1}{2\pi i}\int_{\gamma }^{{}}\frac{f_{t}^{\prime }(\xi )}{f_{t}(\xi )-w%
}d\xi =N(f_{t}-w,\gamma )=1.
\end{equation*}%
Note that
\begin{equation*}
\inf \{|w-f_{t}(z)|:w\in \overline{D_{t}},\text{ }|z|=R\}>0,
\end{equation*}%
because $r<R$ and $f_t$ is continuous and univalent in $\mathbb{D}$.
Moreover, by property LC3, we know that $f_{s}\to f_{t}$
uniformly on $\overline{D(0,R)}$ as $s\to t$. This implies the existence of a number $\delta _{0}>0$ such that
\begin{equation*}
\inf \{|w-f_{u}(z)|:w\in \overline{D_{t}},\text{ }|z|=R\}>0,
\end{equation*}%
for all non-negative $u\in \lbrack t-\delta _{0},t+\delta _{0}].$ In
particular, this allows to consider, for every $w\in D_{t}$ and every
non-negative $u\in \lbrack t-\delta _{0},t+\delta _{0}],$ the Argument
Principle formula
\begin{equation*}
\frac{1}{2\pi i}\int_{\gamma }^{{}}\frac{f_{u}^{\prime }(\xi )}{f_{u}(\xi )-w%
}d\xi =N(f_{u}-w,\gamma ).
\end{equation*}

Again, using property LC3 and the Weierstrass Theorem, we conclude
that
\begin{equation*}
\lim_{u\rightarrow t}\sup \left\{ |N(f_{u}-w,\gamma )-1|:w\in D_{t}\right\}
=0.
\end{equation*}%
But $N(f_u-w,\gamma)$ can take only integer values, so there exists $\delta_1\in(0,\delta _{0})$ such that
\begin{equation*}
\sup \left\{ |N(f_{u}-w,\gamma )-1|:w\in D_{t}\right\} =0,
\end{equation*}
whenever $u\in \lbrack t-\delta _{1},t+\delta _{1}]\cap \lbrack 0,+\infty )$.
In other words, we have showed that
\begin{equation*}
N(f_{u}-w,\gamma )=1,\text{ \ when }u\in \lbrack t-\delta _{1},t+\delta
_{1}]\cap \lbrack 0,+\infty )\text{ and }w\in D_{t}.
\end{equation*}

At this point, we fix $u\in \lbrack t-\delta _{1},t+\delta _{1}]$ and $w\in
D_{t}.$ Our idea is to apply now the generalized Argument Principle for the couple $(id,f_{u}-w)$ and the rectifiable closed curve $\gamma $ (see, e.\,g., \cite[p. 124, chapter V, Theorem 3.6]{Conway}). Namely, recalling
that $f_{u}-w$ is analytic in the unit disk with a unique zero (denoted by $f_{u}^{-1}(w)) $ which is contained in $\mathop{\rm ins}\gamma ,$ we deduce that
\begin{equation*}
\frac{1}{2\pi i}\int_{\gamma }id(\xi )\frac{f_{u}^{\prime }(\xi )}{f_{u}(\xi
)-w}d\xi =id(f_{u}^{-1}(w))N(f_{u}-w,\gamma )=f_{u}^{-1}(w).
\end{equation*}%
In order to finish the proof of Step~1 it is enough to define~$\varepsilon$ as the distance between~$\omega$ and the boundary of $D_t$, which is positive since $D_t$ is open and $\omega\in D_t$ by construction.

\step2 \textit{For any }$r\in (0,1)$\textit{\ and any }$T>0,$\textit{\ we
have that }%
\begin{equation*}
\sup \left\{ |(f_{t}^{-1}\circ f_{s})(z)|:0\leq s\leq t\leq T,\text{ }%
|z|\leq r\right\} <1.
\end{equation*}
\estep

Fix $r\in(0,1)$\textit{\ }and $T>0$ and suppose that the above supremum is $%
1.$ Then, there exist  sequences $(s_{n}),(t_{n})$ and $(z_{n})$ such that:
\begin{itemize}
\item[$(a)$] for all $n\in\mathbb{N}$, $0\leq s_{n}\leq t_{n}\leq T,$ $|z_{n}|\leq
r,$
\item[$(b)$] the following limits exist $s:=\lim_{n}s_{n},$ $t:=\lim_{n}t_{n},$ $z_{0}:=\lim_{n}z_{n},$\\ $%
\beta:=\lim_{n}(f_{t_{n}}^{-1}\circ f_{s_{n}})(z_{n}),$ and
\item[$(c)$] $0\leq s\leq t\leq T,$ $|z_{0}|\leq r,$ $\beta\in\partial\mathbb{D}$.
\end{itemize}
We note that $f_{s}(z_{0})\in \Omega _{s}\subset \Omega _{t}$ and $%
\lim_{n}f_{s_{n}}(z_{n})=f_{s}(z_{0}).$ Therefore, by [Step 1], there exist $%
\varepsilon>0,$ $\delta >0$\ and a Jordan curve $\gamma $\ with $\gamma \cup \mathop{\rm %
ins}\gamma \subset \mathbb{D}$\ such that
\begin{equation*}
f_{u}^{-1}(w)=\frac{1}{2\pi i}\int_{\gamma }\frac{\xi f_{u}^{\prime }(\xi )}{%
f_{u}(\xi )-w}d\xi ,
\end{equation*}%
\text{ whenever }$u\in \lbrack t-\delta ,t+\delta ]\cap[0,+\infty)$\text{
and }$w\in D(f_{s}(z_{0}),\varepsilon)$. In particular, for $n$ large enough, we have that
\begin{equation*}
f_{t_{n}}^{-1}(f_{s_{n}}(z_{n}))=\frac{1}{2\pi i}\int_{\gamma }\frac{\xi
f_{t_{n}}^{\prime }(\xi )}{f_{t_{n}}(\xi )-f_{s_{n}}(z_{n})}d\xi .
\end{equation*}%
Clearly, by property LC3 and the above formula
\begin{equation*}
f_{t_{n}}^{-1}(f_{s_n}(z_{n}))=\frac{1}{2\pi i}\int_{\gamma }\frac{\xi
f_{t_{n}}^{\prime }(\xi )}{f_{t_{n}}(\xi )-f_{s_{n}}(z_{n})}d\xi \longrightarrow \frac{1}{2\pi i}\int_{\gamma }\frac{\xi f_{t}^{\prime }(\xi
)}{f_{t}(\xi )-f_{s}(z_{0})}d\xi =f_{t}^{-1}(f_{s}(z_{0}))
\end{equation*}%
as $n\to+\infty$. Since $f_{t}^{-1}(f_{s}(z_{0}))\in \mathbb{D}$, we obtain a contradiction, which finishes the proof
of Step~2.

\step3 \textit{Let }$\gamma\,:[a,b]\rightarrow\mathbb{C}$\textit{\ be a rectifiable curve in
}$\mathbb{D}$\textit{\ and }$T>0.$\textit{\ Then,
for all }$t\in\lbrack0,T],$\textit{\ the curve }%
\[
\gamma_{t}:[a,b]\rightarrow\mathbb{C}\text{, }\xi\mapsto f_{t}(\gamma(\xi
))\in\Omega_{t}%
\]
\textit{is a well-defined rectifiable curve in }$\Omega_{t}$\textit{.
Moreover,}%
\[
\sup\{\operatorname*{len}(\gamma_{t}):t\in\lbrack0,T]\}<+\infty,
\]
\textit{where, as usual, }$\operatorname*{len}(\gamma_{t})$\textit{\ denotes the length of }$\gamma_{t}.$
\estep

The fact that $\gamma_{t}$ is a well-defined rectifiable curve is widely known. So, suppose that the above
supremum is $+\infty.$ In this case, there
exists a sequence $(t_{n})$ in the interval $[0,T]$ such that $\lim_{n}%
t_{n}=t\in\lbrack0,T]$ and lim$_{n}$ $\operatorname*{len}(\gamma_{t_{n}%
})=+\infty.$ However, the well-known estimate
\[
\operatorname*{len}(\gamma_{t_{n}})\leq\operatorname*{len}(\gamma
)\max\{|f_{t_{n}}^{\prime}(\xi)|:\xi\in\gamma\},
\]
shows (recall that $\gamma$ is a compact set) that there exists a subsequence $(z_{n_{k}})$ in the curve
$\gamma$ converging to some $z_{0}\in\gamma$ such that $\lim_{k}f_{t_{n_{k}}}^{\prime}(z_{n_{k}})=\infty.$
However, by property LC3 and Weierstrass' Theorem, we deduce $\lim_{k}f_{t_{n_{k}}}^{\prime
}(z_{n_{k}})=f_{t}^{\prime}(z_{0})$, obtaining in this way a contradiction.

\step4 \textit{In this step we will finally prove the theorem.} \estep

By properties LC1 and LC2,
we see that the functions
\[
\varphi_{s,t}(z):=f_{t}^{-1}(f_{s}(z)),\text{ }z\in\mathbb{D},\text{ }0\leq s\leq t
\]
are well-defined and, indeed, $\varphi_{s,t}\in\mathrm{Hol}(\mathbb{D},\mathbb{D}),$ for any $0\leq s\leq t$.
Hence, $(\varphi_{s,t})$ will be an evolution family of order $d$ if we are able to prove properties EF1, EF2,
and EF3. The first two properties follow easily from the way we have defined the family
$(\varphi_{s,t}).$ The third property is more difficult to prove. We fix $z\in\mathbb{D}$ and $T>0.$
By [Step 2], there exists $R_{1}%
:=R_{1}(z,T)\in(0,1)$ such that
\[
\sup\{|\varphi_{a,b}(z)|:0\leq a\leq b\leq T\}\leq R_{1}.
\]
Applying again [Step 2], we obtain another $R_{2}:=R_{2}(z,T)\in(0,1)$ such that $R_{2}>R_{1}$
\[
\sup\{|\varphi_{a,b}(z)|:0\leq a\leq b\leq T,|\xi|\leq R_{1}\}<R_{2}.
\]
Additionally, we denote by $\gamma$ the positively oriented circle  of radius $R_{2}$ centered at the origin. As in [Step 3], we also consider the rectifiable  curves $\gamma_{t}:=f_{t}\circ\gamma,$ which are Jordan curves due to the univalence of $f_{t}$.

Now, assume that $0\leq s\leq u\leq t\leq T.$ Then, using property EF2, we obtain%
\[
|\varphi_{s,u}(z)-\varphi_{s,t}(z)|=|\varphi_{s,u}(z)-\varphi_{u,t}%
(\varphi_{s,u}(z))|\leq\sup\{|\varphi_{u,t}(\xi)-\xi|:|\xi|\leq R_{1}\}.
\]
But, for any $|\xi|\leq R_{1},$ we have that $|f_{t}^{-1}(f_{u}(\xi))|<R_{2}$, so $f_{u}(\xi)\in
f_{t}(\operatorname*{ins}\gamma)$. Applying \cite[Lemma 1.1]{Pommerenke}, we see that
$f_{u}(\xi)\in\operatorname*{ins} \gamma_{t}.$ The same argument shows that $f_{t}(\xi)\in\operatorname*{ins}
\gamma_{t}.$ Therefore, using the Cauchy Integral Formula,  for all $|\xi|\leq R_{1}$ we get
\begin{align*}
\left\vert f_{t}^{-1}(f_{u}(\xi))-\xi\right\vert  & =\left\vert f_{t}%
^{-1}(f_{u}(\xi))-f_{t}^{-1}(f_{t}(\xi))\right\vert \\
& =\left\vert \frac{\operatorname*{ind}(\gamma_{t},f_{u}(\xi))}{2\pi i}%
\int_{\gamma_{t}}\frac{f_{t}^{-1}(\eta)}{\eta-f_{u}(\xi)}d\eta-\frac
{\operatorname*{ind}(\gamma_{t},f_{t}(\xi))}{2\pi i}\int_{\gamma_{t}}%
\frac{f_{t}^{-1}(\eta)}{\eta-f_{t}(\xi)}d\eta\right\vert \\
& \le\frac{1}{2\pi}|f_{u}(\xi)-f_{t}(\xi)|\left\vert \int_{\gamma_{t}}%
\frac{f_{t}^{-1}(\eta)}{(\eta-f_{u}(\xi))(\eta-f_{t}(\xi))}d\eta\right\vert .
\end{align*}
We claim that
\[
d=d(z,T):=\inf\{|f_{t}(a)-f_{u}(b)|:0\leq u\leq t\leq T,\text{ }%
|a|=R_{2},\text{ }|b|\leq R_{1}\}>0.
\]
Therefore, recalling that $f_{t}^{-1}(\Omega_{t})\subset\mathbb{D}$ and using the above estimation$,$ we have
\[
\left\vert f_{t}^{-1}(f_{u}(\xi))-\xi\right\vert \leq\frac{1}{2\pi}|f_{u}%
(\xi)-f_{t}(\xi)|\frac{1}{d^{2}}\operatorname*{len}(\gamma_{t}).
\]
Now, by [Step 3], there exists $C=C(z,T)>0$ such that
\[
\sup\{\operatorname*{len}(\gamma_{t}):t\in\lbrack0,T]\}\leq C,
\]
so
\[
|\varphi_{s,u}(z)-\varphi_{s,t}(z)|\leq\frac{C}{2\pi d^{2}}\sup\{|f_{u}%
(\xi)-f_{t}(\xi)|:|\xi|\leq R_{1}\}.
\]
Finally, by property LC3 with $K:=\overline{D(0,R_{1})}$, there exists a non-negative function $k_{z,T}\in
L_{{}}^{d}([0,T];\mathbb{R})$  such that
\[
|\varphi_{s,u}(z)-\varphi_{s,t}(z)|\leq\frac{C}{2\pi d^{2}}\int_{u}^{t}%
k(\eta)d\eta.
\]

Now it remains to prove that $d>0$. Suppose on the contrary that $d=0.$ Then, there exist sequences $(a_{n}),(b_{n}),(u_{n})$
and $(t_{n})$ such that:

$(a)$ for all $n\in\mathbb{N}$, $0\leq u_{n}\leq t_{n}\leq T,$ $|a_{n}%
|=R_{2},$ $|b_{n}|\leq R_{1},$

$(b)$ there exist the following limits $u:=\lim_{n}u_{n},$ $t:=\lim_{n}t_{n},$ $a:=\lim_{n}a_{n},$ $b:=\lim _{n}$ $b_{n},$

$(c)$ $0\leq u\leq t\leq T,$ $|a|=R_{2},$ $|b|\leq R_{1}$, and

$(d)$ $f_{t_{n}}(a_{n})-f_{u_{n}}(b_{n})\to0$ as $n\to+\infty$.

By property LC3, we know that $(f_{u_{n}})$ and $(f_{t_{n}})$ tends to $f_{u}
$ and $f_{t}$, respectively, in the compact-open topology of $\mathrm{Hol}%
(\mathbb{D},\mathbb{C}).$ Therefore, by $(b)$ and $(d),$ we conclude that $f_{u}(b)=f_{t}(a).$ However, using $(c)$
from the definition of the Jordan curves $\gamma$ and $\gamma_{t}$  it is clear  that $a\in\gamma$ and
$f_{t}(a)\in f_{t}\circ\gamma=\gamma_{t}.$ At the same time, $|b|\leq R_{1} $. So by the choice of $R_{2}$ we
find that $|f_{t}^{-1}(f_{u}(b))|<R_{2}$. Thus, $f_{u}(b)\in f_{t}(\operatorname*{ins}\gamma)=\operatorname*{ins}
\gamma_{t}$ by \cite[Lemma 1.1]{Pommerenke}. Obviously
$\gamma_{t}\cap\operatorname*{ins}\gamma_{t}=\emptyset$, so we have a contradiction, which finishes the proof.
\end{proof}

The following lemma shows that if an evolution family has order $d\in[1,+\infty]$, then any Loewner chain
associated with it is also of order~$d$. From another point of view, the next lemma shows that the algebraic equation (\ref{main_EV_LC}) implies indirectly conditions LC2 and LC3.

\begin{lemma}
\label{equationimpliesLC} Let  $(\varphi_{s,t})$ be an evolution family of order $d\in[1,+\infty]$. Assume that for all $t\geq0$ the function $f_t:\D\to\C $ is univalent and
$$
f_t\circ\v_{s,t}=f_s,\quad   0\leq s\leq t<+\infty.
$$
Then the family $(f_t)$ is a Loewner chain of order $d$.
\end{lemma}
\begin{proof} Let $K$ be a compact subset of $\D$ and $T>0$. By Lemma \ref{EF-boundedness},
there exists $R_1\in(0,1)$ such that $|\varphi_{s,t}(z)|\leq R_1$ for all $z\in K$ whenever $0\leq s\leq t\leq T$.
Write $R_2=(1+R)/2$. Again by Lemma \ref{EF-boundedness}, there exists $R_3\in(0,1)$ such that $|\varphi_{s,t}(z)|\leq
R_3$ for all $|z|=R_2$ and  all $0\leq s\leq t\leq T$. Since the function $f_T$ is continuous, there is a
positive constant $M$ such that
\[
|f_t(\xi)|=|f_T(\varphi_{t,T}(\xi))|\leq M
\]
for all $t\leq T$ and any complex number $\xi$ with $|\xi|=R_2$. Fix $z\in K$ and $0\leq s\leq t\leq T$. We
have
\begin{eqnarray*}
  f_s(z)-f_t(z) &=& f_t(\varphi_{s,t}(z))-f_t(z)  \\
   &=& \frac{1}{2 \pi i} \int_{C(0,R_2)^+}\left( \frac{f_t(\xi)}{\xi-\varphi_{s,t}(z)}-
   \frac{f_t(\xi)}{\xi-z}\right)d\xi\\
   &=& \frac{1}{2 \pi i} \int_{C(0,R_2)^+}
   \left( \frac{f_t(\xi)(\varphi_{s,t}(z)-z)}{(\xi-\varphi_{s,t}(z))(\xi-z)}\right) d\xi \\
   &=& \frac{\varphi_{s,t}(z)-z}{2 \pi i} \int_{C(0,R_2)^+}
   \left( \frac{f_t(\xi)}{(\xi-\varphi_{s,t}(z))(\xi-z)}\right) d\xi .
\end{eqnarray*}
Therefore,
\begin{eqnarray*}
  |f_s(z)-f_t(z)|&=& \left|\frac{\varphi_{s,t}(z)-z}{2 \pi i} \int_{C(0,R_2)^+}
   \left( \frac{f_t(\xi)}{(\xi-\varphi_{s,t}(z))(\xi-z)}\right) d\xi \right|\\
              &\leq & R_2\frac{M}{(R_2-R_1)^2}|\varphi_{s,t}(z)-z|
  \end{eqnarray*}
  for all $z\in K$ and $0\leq s\leq t\leq T$.
  Now the conclusion of the lemma easily follows from the last inequality.
\end{proof}

Now we prove the existence of a Loewner chain associated with a given evolution family.

\begin{theorem}\label{Gum_exist}
Let $\fami\varphi$ be an evolution family of order~$d\in[1,+\infty]$. Then there exists a normalized Loewner chain $\famc f$
of order~$d$ associated with the evolution family~$\fami\varphi$ such that the set
$\Omega:=\cup_{t\gtz}f_t(\UD)$ coincides with the disk $\{z:|z|<1/\beta\}$ if $\beta>0$ and with the whole
complex plane~$\C$ if $\beta=0$, where
$\beta=\lim_{t\to+\infty}\frac{|\varphi_{0,t}'(0)|}{1-|\varphi_{0,t}(0)|^2}.$
\end{theorem}
\begin{proof}
By Proposition~\ref{EF-decomposition} we have $\varphi_{s,t}=h_t\circ\psi_{s,t}\circ h_s^{-1}$, where
$(\psi_{s,t})$ is an evolution family such that $\psi_{s,t}(0)=0$ and $\psi'_{s,t}(0)>0$ for all $t\gt s \gtz$,
and $h_t$ is a conformal automorphism of~$\UD$ for each $t\gtz$, with $h_0$ being the identity map.

Now we build the Loewner chain for the evolution family $(\psi_{s,t})$ and then a simple argument will allow us to finish the proof.

By Theorems \ref{EF-bijection-VF} and \ref{Berkson-Porta-for-Herglotz}, there exist a measurable function
$\tau:[0,+\infty)\to \oD$  and a Herglotz function $p(z,t)$ of order $d$ such that for all $z\in \D$ and all
$s\ge0$,
\begin{equation}\label{Gum_LKG} \frac{\de \psi_{s,t}(z)}{\de
t}=(\psi_{s,t}(z)-\tau(t))(\overline{\tau(t)}\psi_{s,t}(z)-1)p(\psi_{s,t}(z),t)
 \quad \hbox{a.e. $t\in
[0,+\infty)$}.
\end{equation}
Since $\psi_{s,t}(0)=0$, $\psi'_{s,t}(0)>0$, $t\gt s\gtz$, we conclude that $\tau(t)\equiv0$. In this case, one can rewrite equation~\eqref{Gum_LKG} in the form
\begin{equation}\label{Gum_LKM}
\frac{\de \psi_{s,t}(z)}{\de t}=-\psi_{s,t}(z)p(\psi_{s,t}(z),t).
\end{equation}

We will
show that the functions
\begin{equation}\label{Gum_lim}
g_s(z):=\lim_{t\to+\infty} \frac{\psi_{s,t}(z)}{\psi'_{0,t}(0)},
\end{equation}
where the limit is attained uniformly on compact subsets of the unit disk, form a Loewner chain associated with $(\psi_{s,t})$. Our proof of the existence of that limit follows the approach given in \cite[Chapter 6]{Pommerenke}. However,
for the sake of clearness and completeness, we include the details.

Assume for a moment that such a limit does exist. Then $g_s'(0):=\lim_{t\to+\infty}
\frac{\psi_{s,t}'(0)}{\psi'_{0,t}(0)}=\frac{1}{\psi'_{0,s}(0)}>0$. Moreover, since all the functions
$\psi_{s,t}$ are univalent \cite[Corollary 6.3]{BCM1}, we conclude that the function $g_s$ is univalent for all $s\ge0$. Moreover, by construction
$$
g_t\circ\psi_{s,t}(z)=\lim_{u\to+\infty} \frac{\psi_{t,u}(\psi_{s,t}(z))}{\psi'_{0,u}(0)} = \lim_{u\to+\infty}
\frac{\psi_{s,u}(z)}{\psi'_{0,u}(0)}=g_s(z),\qquad 0\leq s\leq t<+\infty.
$$
Therefore, by Lemma \ref{equationimpliesLC}, the family $(g_t)$ is a Loewner chain of order $d$ associated with
$(\psi_{s,t})$. Also, it is clear that it is a normalized Loewner chain.

Therefore, we have only to prove the existence of~\eqref{Gum_lim}.

By \cite[Proof of Theorem 7.1]{BCM1}, for all $z\in\mathbb{D}$ and $t>s\ge0,$
\begin{equation}\label{formav}
\psi_{s,t}(z)=z\exp\left( -\int_{s}^{t}p(\psi_{s,\xi}(z),\xi )d\xi\right).
\end{equation}
Write $\Lambda_{s,t}(z):=\int_s^t\left(p(0,\xi)-p(\psi_{s,\xi}(z),\xi)\right)d\xi$. Notice that
$$\psi_{s,t}'(0)=\exp\left(-\int_s^t p(0,\xi)d\xi\right)>0.$$ Therefore,
\begin{equation}
\frac{\psi_{s,t}(z)}{\psi_{0,t}'(0)}=z\exp\left( \int_{0}^{s}p(0,\xi )d\xi\right)
\exp\left(\Lambda_{s,t}(z)\right) .
\end{equation}
Now in order to prove the existence of the limit~\eqref{Gum_lim}, it is sufficient
to show that $\Lambda_{s,t}$ has a limit as $t\to+\infty$ which is attained uniformly on compact subsets of the unit disk.

By property EF2, we have that $\psi_{0,t}'(0)=\psi_{s,t}'(0)\psi_{0,s}'(0)\leq\psi_{0,s}'(0)$, because
$\psi_{s,t}(0)=0$ and $\psi_{s,t}(\D)\subseteq \D$. That is
\begin{equation*}
\frac{\partial \psi_{0,t}'(0)}{\partial t}\leq 0,\quad \text{ a.e. \ } t\in[0,+\infty).
\end{equation*}
Since
$$
\frac{\partial \psi_{0,t}'(0)}{\partial t}=-p(0,t)\exp\left( -\int_{s}^{t}p(0,\xi
)d\xi\right)=-p(0,t)\psi_{s,t}'(0),~~~~
 \text{  a.e. \ } t\in[0,+\infty),
$$
we conclude that $p(0,t)\geq 0$ for a.e. $t\in [0,\infty)$.

When $\Re p(\cdot,\xi)>0$ (otherwise, $p(\cdot,\xi)$ is constant), necessarily $p(0,\xi)> 0$ and the holomorphic map $z\mapsto\frac{p(z,\xi)-p(0,\xi)}{p(z,\xi)+\overline{p(0,\xi)}}
$ sends the unit disc into itself and fixes the origin. Then
\begin{align*}
\left\vert p(z,\xi)-p(0,\xi)\right\vert  & \leq|z|\left\vert p(z,\xi
)+\overline{p(0,\xi)}\right\vert \leq|z||p(z,\xi)|+|z||p(0,\xi)|\\
& \leq|z|\frac{1+|z|}{1-|z|}|p(0,\xi)|+|z||p(0,\xi)|=\frac{2\left\vert z\right\vert }{1-\left\vert z\right\vert
} p(0,\xi) ,
\end{align*}
where we have used~\cite[pages 39-40]{Pommerenke}. Therefore, by \cite[Theorem
1.6]{Pommerenke}, we have
\begin{eqnarray*}
  \left\vert p(\psi_{s,\xi}(z),\xi)-p(0,\xi)\right\vert & \leq &
 \frac{2\left\vert \psi_{s,\xi}(z)\right\vert
}{1-\left\vert \psi_{s,\xi}(z)\right\vert }p(0,\xi) \leq  \frac{2\left\vert \psi_{s,\xi}(z)\right\vert
}{1-\left\vert z\right\vert } p(0,\xi)  \\
   &\leq & \frac{2\left\vert \psi_{s,\xi}'(0)\right\vert
}{(1-\left\vert z\right\vert)^3 } p(0,\xi)
= \frac{2\exp \left( -\int_s^\xi p(0,u)du\right) }{(1-\left\vert
z\right\vert)^3 } p(0,\xi) .
\end{eqnarray*}
Now, we can bound the function $\Lambda_{s,\cdot}(z)$:
\begin{eqnarray*}\label{existence_2}
 \left| \Lambda_{s,t}(z)-\Lambda_{s,u}(z)\right|&\leq &
 \int_u^t\left|p(0,\xi)-p(\psi_{s,\xi}(z),\xi)\right| d\xi \\
   &\leq &  \frac{2}{(1-\left\vert z\right\vert)^3 }
\int_u^t \exp \left( -\int_0^\xi p(0,u)du\right)p(0,\xi) d\xi \\
   &=& \frac{2}{(1-\left\vert z\right\vert)^3 }
\int_u^t \frac{\partial }{\partial \xi}\left[-\exp \left( -\int_0^\xi p(0,u)du\right) \right] d\xi \\
   &=& \frac{2}{(1-\left\vert z\right\vert)^3 }
   \left(\exp\left(-\int_0^up(0,\xi)d\xi\right)-\exp\left(-\int_0^tp(0,\xi)d\xi\right)\right).
\end{eqnarray*}
Finally, from these last inequalities and the fact that $$\lim_{t\to+\infty}\exp\left(-\int_0^t
p(0,\xi)d\xi\right)\in[0,1]$$ (recall that $p(0,\xi)\ge0$ for a.e. $\xi\in[0,+\infty)$), we conclude that the limit
\eqref{Gum_lim} does exist.

Now we consider the family $f_t=g_t\circ h_t^{-1}$. It easy to see that $\famc f$ satisfies the hypothesis of
Lemma~\ref{equationimpliesLC} and hence it is a Loewner chain of order~$d$ associated with~$\fami\varphi$. Since
$f_0=g_0$, the Loewner chain $(f_t)$ is normalized.

Now, let us describe the set $\Omega=\cup_{t\gtz}f_t(\UD)=\cup_{t\gtz}g_t(\UD)$.
An easy computation shows $\psi_{0,t}'(0)=\frac{|\varphi_{0,t}'(0)|}{1-|\varphi_{0,t}(0)|^2}$. In particular,
 since the map  $t\mapsto \psi_{0,t}'(0)$ is monotone, the number
$$
\beta=\lim_{t\to+\infty}\frac{|\varphi_{0,t}'(0)|}{1-|\varphi_{0,t}(0)|^2}
$$
is well-defined.

In view of the equality $g_t'(0)=1/\psi_{0,t}'(0)$,  Koebe's theorem shows that $g_t(\UD)$ contains
a disk of radius $1/(4\psi_{0,t}'(0))$ centered at the origin. In particular, if $\beta=0$, then  $\cup_{t\gtz}g_t(\UD)=\C$.

Suppose now that $\beta>0$. We have proved that in this case $\psi_{s,t}$ has a limit $\psi_s$ as $t\to+\infty$.
Note that $\psi_s'(0)=\beta/\psi'_{0,s}(0)\to 1$ as $s\to+\infty$, while $\psi_s(\UD)\subset\UD$ and
$\psi_s(0)=0$. It follows that $\psi_s\to\id_\UD$ as $s\to+\infty$. Then $g_s$ tends to the mapping $z\mapsto
z/\beta$ as $s\to+\infty$ locally uniformly in $\UD$. Since $g_s(\UD)$ forms an increasing family of domains, it
follows that $\cup_{s\gtz} g_s(\UD)=\{z:|z|<1/\beta\}$.

The proof is now finished.
\end{proof}

In the above proof we have obtained that the function $\beta :[0,+\infty )\rightarrow (0,1]$ given by
\begin{equation*}
\beta(t):=\frac{1}{1-|\varphi _{0,t}(0)|^{2}}|\varphi _{0,t}'(0)|\qquad \text{for all }t\geq 0,\text{ }z\in \D
\end{equation*}
is non-increasing and, as a consequence, the following limit exist
\begin{equation*}
\beta :=\lim_{t\rightarrow +\infty }\beta (t)\in [0,1].
\end{equation*}
This number will play a crucial role in the study of uniqueness of Loewner chains associated with the evolution
family $(\varphi_{s,t})$. For this reason, in the next proposition we analyze in full generality the above
limit.

\begin{proposition}\label{beta_prop}
Let $(\varphi _{s,t})$ be an evolution family of order $d\in[1,+\infty]$ and define
\begin{equation*}
\beta _{z}(t):=\frac{1-|z|^{2}}{1-|\varphi _{0,t}(z)|^{2}}|\varphi _{0,t}^{\prime }(z)| \qquad \text{for all
}t\geq 0,\text{ }z\in \D.
\end{equation*}
Then
\begin{enumerate}
\item For all $z\in \D,$ the map $\beta (z):[0,+\infty )\rightarrow
(0,1]$ is absolutely continuous and non-increasing. In particular,
there exists the following limit
\begin{equation*}
\beta (z):=\lim_{t\rightarrow +\infty }\beta _{z}(t).
\end{equation*}

\item The following assertions are equivalent:

\begin{enumerate}
\item There exists $z\in \D$ such that $\beta (z)=0.$

\item For all $z\in \D$ we have  $\beta (z)=0.$
\end{enumerate}

\item The following assertions are equivalent:

\begin{enumerate}
\item There exists $z\in \D$ with  $\beta (z)=1.$

\item For all $z\in \D,$ we have  $\beta (z)=1.$

\item For all $t\geq 0,$ the map $\varphi _{0,t}$ is an automorphism.

\item For all $0\leq s\leq t,$ the map $\varphi _{s,t}$ is an automorphism.
\end{enumerate}

\item If there is $z\in \D$ such that  $\beta (z)<1,$ then there is $
T\in \lbrack 0,+\infty )$ such that $\varphi _{0,t}$ is an automorphism for all $0\leq t\leq T$ and
$\varphi_{0,t}$ is not an automorphism for all $t>T.$
\end{enumerate}
\end{proposition}
\begin{proof}
By Proposition \ref{EF-decomposition} we have $\varphi_{s,t}=h_t\circ \psi_{s,t}\circ h_s^{-1}$, where
$(\psi_{s,t})$ is an evolution family such that $\psi_{s,t}(0)=0$ and $\psi'_{s,t}(0)>0$ for all $t\geq 0$,
 and $h_t$ is a conformal automorphism of $\D$ for each
$t\geq 0$, with $h_0$ being the identity map.

One can check that
\begin{equation*}
\beta _{z}(t):=\frac{1-|z|^{2}}{1-|\varphi _{0,t}(z)|^{2}}|\varphi _{0,t}^{\prime }(z)|=\frac{1-|z|^{2}}{1-|\psi
_{0,t}(z)|^{2}}\psi _{0,t}^{\prime }(z)\qquad \text{for all }t\geq 0,\text{ }z\in \D.
\end{equation*}

\noindent{\it Proof of (1).} The absolute continuity of the function $\beta _{z}$ is just an easy consequence of
Proposition~\ref{EF-properties}.

Denote by $\widetilde{\rho }_\D$ the pseudo-hyperbolic distance in the unit disk. Since any holomorphic self-map of
the unit disk is a contraction for $\widetilde{\rho }_{\D},$ given $s<t$ and $z,w\in \D,$ we have
\begin{equation*}
\widetilde{\rho }_{D}(\varphi _{0,t}(w),\varphi _{0,t}(z))=%
\widetilde{\rho }_{D}(\varphi _{s,t}(\varphi _{0,s}(w)),\varphi _{s,t}(\varphi _{0,s}(z)))\leq \widetilde{\rho
}_{D}(\varphi _{0,s}(w),\varphi _{0,s}(z)).
\end{equation*}
That is
\begin{equation*}
\left\vert \frac{\varphi _{0,t}(w)-\varphi _{0,t}(z)}{1-\overline{\varphi _{0,t}(w)}\varphi
_{0,t}(z)}\right\vert \leq \left\vert \frac{\varphi _{0,s}(w)-\varphi _{0,s}(z)}{1-\overline{\varphi
_{0,s}(w)}\varphi _{0,s}(z)} \right\vert .
\end{equation*}
Dividing by $|w-z|$ $(w\neq z)$ and taking limits as $w\to z,$ we deduce that
\begin{equation*}
\frac{|\varphi _{0,t}^{\prime }(z)|}{1-|\varphi _{0,t}(z)|^{2}}\leq \frac{ |\varphi _{0,s}^{\prime
}(z)|}{1-|\varphi _{0,s}(z)|^{2}}.
\end{equation*}
Thus $\beta _{z}(t)\leq \beta _{z}(s)$ for all $0\leq s<t<+\infty .$

\noindent{\it Proof of (2).} Notice that we know that the number $\beta (0)=\lim_{t\rightarrow +\infty }\psi _{s,t}^{\prime }(0)$ is well
defined. Moreover, the family of functions $(\psi _{0,t})_{t\geq 0}$ is normal. So there is a sequence $
(t_{n})\rightarrow +\infty $ such that the limit $f(z)=\lim_{n}\psi _{0,t_{n}}(z)$ exists for all $z\in\UD$ and it is attained uniformly on compact subsets of $\UD$. The function $f$ is either constant or univalent in $\UD$, with $f(0)=0$ and  $f^{\prime }(0)=\beta (0).$ Therefore $f$ vanishes identically if and only if
$\beta (0)=0.$ Otherwise, $f$ is univalent and $f^{\prime }(z)\neq 0$ for all $z\in\UD$. Now, observe
that
$$
\beta (z)=\lim_{t\rightarrow +\infty }\beta _{z}(t)=\lim_{n\rightarrow +\infty }\beta
_{z}(t_{n})=\lim_{n\rightarrow +\infty }\frac{1-|z|^{2}}{1-|\psi _{0,t_{n}}(z)|^{2}}|\psi _{0,t_{n}}^{\prime
}(z)|=\frac{1-|z|^{2}}{1-|f(z)|^{2}}|f^{\prime }(z)|.$$
 That is, $\beta (z)=0$ for some $z\in\D$ if and only if $f^{\prime }(z)=0$ for some $z\in\D$
 if and only if $f$ is zero (recall that $f(0)=0$).

Assertions (3) and (4) are easy and we leave their proofs to the reader.
\end{proof}

\begin{definition}
Let $\varphi_{s,t}$ be an evolution family and take
$\beta=\lim_{t\to+\infty}\frac{|\varphi_{0,t}'(0)|}{1-|\varphi_{0,t}(0)|^2}$.
 Let $(f_t)$ be a normalized Loewner chain associated with
$\varphi_{s,t}$. We say that $(f_t)$ is a {\it standard Loewner chain} if $\cup_{t\gtz}f_t(\UD)=\{z:|z|<1/\beta\}$
(obviously, when $\beta=0$, by $\{z:|z|<1/\beta\}$ we mean the complex plane $\C$).
\end{definition}

Note that if $(f_t)$ is a Loewner chain associated with a given evolution family~$(\v_{s,t})$ and $h$ is any univalent holomorphic function in $\Omega:=\cup_{t\ge0}f_t(\UD)$, then the formula $g_t=h\circ f_t$, $t\ge0$, defines a Loewner chain which is also associated with $(\v_{s,t})$. In view of this remark, the following theorem gives a necessary and sufficient condition for an evolution family to have a unique normalized Loewner chain associated with it. Moreover, in case of non-uniqueness, the set of all normalized Loewner chains associated with $(\v_{s,t})$ is explicitly described.

As usual,  we denote by  $\clS$ the class of all univalent holomorphic functions $h$ in the unit disk~$\UD$,
normalized by $h(0)=h'(0)-1=0$. As above, $\beta=\lim_{t\to+\infty}|\varphi_{0,t}'(0)|/(1-|\varphi_{0,t}(0)|^2)$.

\begin{theorem}\label{Gum_uniq}
Let $\fami\varphi$ be an evolution family.
\begin{enumerate}
  \item There is a unique standard Loewner chain $\famc f$ associated with~$\fami\varphi$.
  \item If $\beta =0$, then there is a unique normalized Loewner chain $\famc f$ associated with~$\fami\varphi$ (and obviously, it is the standard one.)
  \item If $\beta>0$ and $(g_t)$ is a normalized Loewner chain associated with $\fami\varphi$, then there is $h\in \clS$ such that
    \begin{equation}\label{allLC}
g_t(z)=h\big(\beta f_t(z)\big)/\beta,
\end{equation}
where $(f_t)$ is the unique standard Loewner chain  associated with~$\fami\varphi$.
\end{enumerate}
\end{theorem}
\begin{proof}
Let $(f_t)$ be the standard Loewner chain built in Theorem \ref{Gum_exist} and  $(g_t)$ another normalized
Loewner chain associated with the evolution family $(\varphi_{s,t}).$ For each $t\geq 0$ denote by
$k_t:f_t(\D)\to g_t(\D)$ the function $k_t=g_t\circ f_t^{-1}$. Write $\Omega _1=\cup_{t\geq0}
f_t(\D)=\{z:|z|<1/\beta\}$ and $\Omega _2=\cup_{t\geq0} g_t(\D).$

If $s<t$ and $w\in f_s(\D)$ with $w=f_s(z)$, we have that
$$
k_t(w)=g_t\circ f_t^{-1}(f_s(z))=g_t\circ
f_t^{-1}(f_t(\varphi_{s,t}(z)))=g_t(\varphi_{s,t}(z))=g_s(z)=g_s(f_s^{-1}(w))=k_s(w).
$$
That is, $k_{t|f_s(\D)}=k_s$. This property says that the function $k:\Omega_1\to\Omega_2$ defined by
$k(w):=k_t(w)$ for some (or any) $t$ such that $w\in f_t(\D)$ is well-defined, univalent and onto. Moreover
$k(0)=0$ and $k'(0)=1$. Notice that $k\circ f_t=g_t$ for all $t$.

Now suppose that $\beta=0$. Then $\Omega_1=\C$. Since $\Omega_2$ is a simply connected domain biholomorphic to
$\C$, we also have that $\Omega_2=\C$. In this case, $k$ is a univalent entire function such that
$k(0)=k'(0)-1=0$. Then $k$ is the identity and $f_t=g_t$ for all $t$. This implies statement (2) and statement
(1) for the case $\beta=0$.

If $\beta>0$, denote by $h:\D\to\Omega_2$ the function $h(z)=\beta k(z/\beta)$. Obviously, $h$ belongs to~$\clS$ and satisfies~\eqref{allLC}. This proves statement (3). Finally, if $(g_t)$ is also a
standard Loewner chain associated with $(\varphi_{s,t})$, then $\Omega_2=\{z:|z|<1/\beta\}$. In this case
$k:\{z:|z|<1/\beta\}\to \{z:|z|<1/\beta\}$ is biholomorphic and $k(0)=k'(0)-1=1$. That is $k$ is the identity
and $f_t=g_t$ for all $t\ge0$. This proves statement (1) for $\beta>0$.

The proof is now complete.
\end{proof}

\begin{remark}
 It is clear from the above proof that one can define the standard Loewner chain as the unique normalized Loewner chain $(f_t)$, associated with the evolution family~$(\v_{s,t})$, such that $\cup_{t\ge0}f_t(\UD)$ is either an Euclidean disk or the whole complex plane.
\end{remark}

Our next theorem says that, in some particular cases, the univalence of the functions which form a
Loewner chain can be replaced by an appropriate bound of these functions on certain hyperbolic disks.

\begin{theorem}\label{no_univ_assum}
Let $(\varphi _{s,t})$ be an evolution family in the unit disk having a unique normalized Loewner chain
associated with it. Suppose $(f_{t})_{t\ge0}$
is a family in $\mathrm{Hol}(\mathbb{D},\mathbb{C%
}).$ Then $(f_{t})$ is the unique normalized Loewner chain associated with~$(\varphi _{s,t})$
if and only if the following three conditions are satisfied:
\begin{enumerate}
\item The function $f_{0}$ is normalized, that is, $f_{0}(0)=f_{0}^{\prime }(0)-1=0.$

\item The equation $f_{t}\circ \varphi _{s,t}=f_{s}$ holds for any $0\leq
s\leq t.$

\item For each $R>0,$ there exists some $C>0$ (independent on $t)$ such that
for all $t\geq 0$ the following inequality
\begin{equation*}
|f_{t}(z)|\leq \frac{C}{\beta (t)},
\end{equation*}%
where
$$\beta(t)=\dfrac{|\varphi _{0,t}^{\prime }(0)|}{1-|\varphi _{0,t}(0)|^{2}},\quad t\geq 0,$$
holds for any $z$ in the hyperbolic disk of radius $R$ centered at $\varphi _{0,t}(0)$.
\end{enumerate}
\end{theorem}

\begin{proof}
Before dealing with the proof of the theorem, we comment (really recall) some facts and notations which will be
used later on and which have been shown in the course of proofs of previous results. In our case, according
to Theorem \ref{Gum_uniq}, we know that $\lim_{t\rightarrow +\infty }\beta (t)=0.$

Write $$a(t)=\varphi _{0,t}(0),~~~ b(t)=\dfrac{\varphi _{0,t}'(0)}{|\varphi _{0,t}'(0)|},~~~
h_{t}(z)=\dfrac{b(t)z+a(t)}{1+b(t)\overline{a(t)}z},~~~\text{and}~~~ h_{t}^{-1}(z)=
\overline{b(t)}\dfrac{z-a(t)}{1-\overline{a(t)}z},$$ for all $z\in \D$ and all $t\geq 0$. Clearly, $a(0)=0$
and $b(0)=1$. Finally, define $\psi _{s,t}=h_{t}^{-1}\circ \varphi _{s,t}\circ h_{s}$ for all $0\leq s\leq t.$
One can easily prove that $\psi_{s,t}(0)=0$ and $\psi'_{s,t}(0)=\beta (t)/\beta (s)>0$ for all $0\leq
s\leq t$. Hence, by Proposition \ref{EF-decomposition}, $(\psi_{s,t})$ is an evolution family.

$(\Rightarrow )$ Assume that $(f_{t})$ is the (unique) normalized Loewner chain associated with $(\varphi
_{s,t})$. By the very definition of normalized Loewner chains, we see that only property (3) requires a proof.
Note that
\begin{equation*}
\frac{|\psi _{0,t}^{\prime }(0)|}{1-|\psi _{0,t}(0)|^{2}}=|\psi _{0,t}^{\prime }(0)|=\beta (t)\to0\quad\text{as}\quad t\to+\infty.
\end{equation*}
Therefore, according to Theorem \ref{Gum_uniq} (now applied to the evolution family $(\psi_{s,t})$), we deduce
that $(\psi _{s,t})$ has also a unique normalized Loewner chain associated with it. Moreover, such a Loewner
chain $(g_{t})$ satisfies the equality $g'_t(0)\psi'_{s,t}(0)=g'_s(0)$. Consequently, $g'_t(0)\beta(t)=g'_s(0)\beta(s)$ for all
$t,s\ge0$. But $g'_0(0)\beta(0)=1$. Thus $g'_s(0)={1}/{\beta(s)}$ for all $s\ge0$. Using the Distortion
Theorem, we conclude that
\begin{equation*}
|g_{s}(z)|\leq \frac{1}{\beta (s)}\frac{|z|}{(1-|z|)^{2}}
\end{equation*}
for all $s\geq 0$ and for all $z\in\D$.

Now, fix $R>0$ and $s\geq 0$ and consider $r=\dfrac{e^{R}-1}{e^{R}+1}\in (0,1).$
Take $z$ in the hyperbolic disk
of radius $R$ centered at the point $a(s)=\varphi _{0,s}(0)$.  We have that
\begin{equation*}
\rho_\D(h_{s}^{-1}(z),0))= \rho_\D(h_{s}^{-1}(z),h_{s}^{-1}(a(s)))= \rho_\D(z,a(s))\leq R.
\end{equation*}%
Thus, $|h_{s}^{-1}(z)|\leq r$ and
\begin{equation*}
|f_{s}(z)|=|g_{s}(h_{s}^{-1}(z))|\leq \frac{1}{\beta (s)}\frac{r}{(1-r)^2}=\frac{e^{2R}-1}{2\beta (s)}.
\end{equation*}

$(\Leftarrow )$ First of all, bearing in mind Lemma~\ref{equationimpliesLC} and property (1) combined with
Theorem~\ref{Gum_uniq}, we see that we only have to prove the univalence of each function $f_{t}.$ We start by
defining
\begin{equation*}
g_{t}:=f_{t}\circ h_{t}\in \mathrm{Hol}(\mathbb{D},\mathbb{C}),\text{ }t\geq 0\text{. }
\end{equation*}%
By property (2), we observe that
\begin{equation*}
g_{t}\circ \psi _{s,t}=g_{s},\text{ }0\leq s\leq t.
\end{equation*}%
We notice that the family $(g_{t})$ satisfies the following three properties:
\begin{itemize}
  \item[(a)] $g_{t}(0)=0,$ for all $t\geq 0.$
  \item[(b)] $g_{t}^{\prime }(0)=\beta (t)^{-1},$ for all $t\geq 0.$
  \item[(c)] For all $R>0,$ there exists some $C>0$ such that, for all $%
t\geq 0$ and all $|z|\leq R,$ we have
\begin{equation*}
|g_{t}(z)|\leq C\beta (t)^{-1}.
\end{equation*}
\end{itemize}

Now, fix $s\geq 0$ and $r\in (0,1)$ and suppose that $|z|\leq r.$ Take also
some $R\in (0,1)$ with $R>r.$ By Schwarz Lemma,%
\begin{equation*}
|\psi _{s,t}(z)|\leq |z|\leq r,\text{ for all }t\geq s.
\end{equation*}
Then by the Cauchy Integral Formula, for all $t\geq s$ we have%
\begin{eqnarray*}
|g_{s}(z)-\beta (t)^{-1}\psi _{s,t}(z)| &=&|g_{t}(\psi _{s,t}(z))-\beta
(t)^{-1}\psi _{s,t}(z)| \\
&=&|g_{t}(\psi _{s,t}(z))-g_{t}(0)-g_{t}^{\prime }(0)\psi _{s,t}(z)| \\
&=&\left\vert \frac{1}{2\pi i}\int_{C^{+}(0,R)}g_{t}(\xi )\frac{(\psi
_{s,t}(z))^{2}}{\xi ^{2}(\xi -\psi _{s,t}(z))}d\xi \right\vert \\
&\leq &\frac{2\pi R}{2\pi R^{2}(R-r)}|\psi _{s,t}(z)|^{2}\max \{|g_{t}(\xi )|:|\xi |\leq R\}.
\end{eqnarray*}%
Therefore, by property (3), we can find $C=C(R)$ (independent on $t$) such that
\begin{equation*}
|g_{s}(z)-\beta (t)^{-1}\psi _{s,t}(z)|\leq \frac{C}{R(R-r)}\beta (t)^{-1}|\psi _{s,t}(z)|^{2}.
\end{equation*}%
In fact, since $\lim_{t\rightarrow +\infty }\beta (t)=0$ and by the
Distortion Theorem, we deduce that%
\begin{eqnarray*}
|g_{s}(z)-\beta (t)^{-1}\psi _{s,t}(z)| &\leq &\frac{C}{R(R-r)}\beta (t)^{-1}|\psi _{s,t}^{\prime
}(0)|^{2}\left( \frac{r}{(1-r)^{2}}\right) ^{2}
\\
&\leq &\frac{C}{R(R-r)(1-r)^{4}}\frac{1}{\beta (t)}\frac{\beta ^{2}(t)}{%
\beta ^{2}(s)} \\
&=&\frac{C}{R(R-r)(1-r)^{4}\beta ^{2}(s)}\beta (t)\to0\quad \text{as}\quad {t\to+\infty }.
\end{eqnarray*}%
Therefore, we conclude that
\begin{equation*}
g_{s}=\lim_{t\rightarrow +\infty }\beta (t)^{-1}\psi _{s,t},
\end{equation*}
in the compact-open topology of $\mathrm{Hol}(\mathbb{D},\mathbb{C}).$ By Hurwitz's Theorem and property (b), we
find that $g_{s}$ is univalent.
Since $h_{t}$ is an automorphism of the disk, we finally conclude that $f_{s}$ is univalent as well. The proof is complete.
\end{proof}

\begin{remark}
The above proof shows that statement (3) in this last theorem can be replaced by ``for all $z\in \D$ and for all
$s\geq 0$, the following inequality holds
\begin{equation*}
|f_{s}\circ h_s(z)|\leq \frac{1}{\beta (s)}\frac{|z|}{(1-|z|)^{2}},
\end{equation*}
where, as usual, $h_{s}(z)=\dfrac{b(s)z+a(s)}{1+b(s)\overline{a(s)}z}$,  $ a(t)=\varphi _{0,a}(0)$, and
$b(s)=\dfrac{\varphi _{0,s}'(0)}{|\varphi _{0,s}'(0)|}$."

\end{remark}

\section{Loewner chains and partial differential equations}\label{diff_eq_sect}

In classical Loewner theory any Loewner chain satisfies the Loewner~-- Kufarev PDE, while the corresponding evolution family satisfies the Loewner -- Kufarev ODE with the same driving term. Now we prove an analogue of this statement in our general setting.

\begin{theorem}\label{diff_eq_t}The following assertions hold.

\begin{enumerate}
\item Let $(f_{t})$ be a Loewner chain of order $d\in[1,+\infty].$ Then

\begin{enumerate}
\item There exists a set $N\subset [ 0,+\infty )$ (not depending on $z$)
of zero measure such that for every $s\in (0,+\infty )\setminus N$ the function
\begin{equation*}
z\in \mathbb{D}\mapsto \frac{\partial f_{s}(z)}{\partial s} :=\lim_{h\rightarrow
0}\frac{f_{s+h}(z)-f_{s}(z)}{h}\in \C
\end{equation*}
is a well-defined holomorphic function on $\D.$

\item There exist a Herglotz vector field $G$ of order $d$ and a set
$N\subset [0,+\infty )$ (not depending on $z$) of zero measure such that for every $s\in (0,+\infty)\setminus
N$ and every $z\in \D$,
\begin{equation*}
\frac{\partial f_{s}(z)}{\partial s}=-G(z,s)f_{s}^{\prime }(z).
\end{equation*}
\end{enumerate}

\item Let $G$ be a Herglotz vector field of order $d\in[1,+\infty]$ associated with
the evolution family $(\varphi _{s,t}).$ Suppose that $(f_{t})$ is a family of univalent holomorphic functions in the unit disk
such that
\begin{equation*}
\frac{\partial f_{s}(z)}{\partial s}=-G(z,s)f_{s}^{\prime }(z)\qquad \text{ for every }z\in \D\text{, a.e. }s\in
\lbrack 0,+\infty ).
\end{equation*}
Then $(f_{t})$ is a Loewner chain of order $d$ associated with the evolution family~$(\varphi _{s,t}).$
\end{enumerate}
\end{theorem}

\begin{proof}[Proof of (1.a).]
By the very definition of Loewner chain, the map $s\in [0,+\infty )\mapsto f_{s}(z)\in \C$ is absolutely
continuous, for all fixed $z\in \D$. Thus there exists a set of zero measure $ N_{1}(z)\subset  [0,+\infty)$
such that for every $s\in \lbrack 0,+\infty )\setminus N_{1}(z)$ the following limit exists
\begin{equation*}
D_{s}(z)=\frac{\partial f_{s}(z)}{\partial s}=\lim_{h\rightarrow 0}\frac{ f_{s+h}(z)-f_{s}(z)}{h}.
\end{equation*}
Let $k_{n}\in L_{loc}^{d}([0,+\infty ),\R)$ be a non negative function such that
\begin{equation*}
|f_{s}(z)-f_{t}(z)|\leq \int_{s}^{t}k_{n}(\xi )d\xi
\end{equation*}
whenever $|z|\leq 1-1/n$ and $0\leq s\leq t.$ For each $n,$ there exists a set $N_{2}(n)\subset [0,+\infty )$ of
zero measure such that for every $s\in [0,+\infty )\setminus N_{2}(t)$ there exists the limit
\begin{equation*}
k_{n}(s)=\lim_{h\rightarrow 0}\frac{1}{h}\int_{s}^{s+h}k_{n}(\eta )d\eta.
\end{equation*}
Let us define
\begin{equation*}
N:=\left( \bigcup_{n=1}^{\infty }N_{1}\left(\frac{1}{n+1}\right)\right) \bigcup \left( \bigcup_{n=1}^{\infty }N_{2}(n)\right).
\end{equation*}
Obviously, $N$ is a subset of $[0,+\infty )$ of zero measure, independent of $z$. We are going to prove that for
all $s\in [ 0,+\infty )\setminus N$ the following limit
\begin{equation*}
\lim_{h\rightarrow 0}\frac{f_{s+h}(z)-f_{s}(z)}{h}
\end{equation*}
exists and attained uniformly on compact subsets of $\D.$

First of all we show that for every $s\in (0,+\infty )\setminus N$ the family
\begin{equation*}
\mathcal{F}_{s}:=\{F_{h}:=\frac{1}{h}(f_{s+h}-f_{s}):0<h<1\text{ or }-s<h<0\}
\end{equation*}
is a relatively compact set in $\mathrm{Hol}(\D,\C)$.
To this aim, we consider two cases: $(a)$ $0<h<1;$ $(b)$
$-s<h<0.$

{\it Case $(a)$}: \ Fix $r\in (0,1).$ Let $n\in \N$ be such that $r<1-1/n$. Then, for every $|z|\leq r,$
\begin{equation*}
|F_{h}(z)|=\left\vert \frac{1}{h}(f_{s+h}(z)-f_{s}(z))\right\vert \leq \frac{ 1}{h}\int_{s}^{s+h}k_{n}(\xi )d\xi
\leq \tilde{C}<+\infty ,
\end{equation*}
where the last inequality takes place since $s\notin N_{2}(n)$. Hence, $$\sup \{|F_{h}(z)|:|z|\leq r,
0<h<1\}<+\infty$$ and consequently, by the Montel criterion, the subfamily of $\mathcal{F}_{s}$ with $0<h<1$ is a normal family in $\UD$, as wanted.

{\it Case $(b)$}: the proof is similar to that of case $(a)$ and we omit it.

Thus the family $\mathcal{F}_{s}$ is relatively compact in $\mathrm{Hol}(\D,\C)$.
Let $\psi ,\phi $ be any pair of
limit functions of $\mathcal{F}_{s}$ as $h\to0$. By the very definition of $N$,
\begin{equation*}
D_{s}\left(\frac{1}{m+1}\right)=\psi\left (\frac{1}{m+1}\right)=\phi \left(\frac{1}{m+1}\right),
\end{equation*}
for every $m\in \N$.
But $\{\frac{1}{m+1}\}$ is a sequence accumulating at $0$, hence by the identity principle
$\psi =\phi $. This shows that
\begin{equation*}
\lim_{h\rightarrow 0}\frac{f_{s+h}(z)-f_{s}(z)}{h},
\end{equation*}
exists for all $s\in (0,+\infty )\setminus N$ and is attained uniformly on compact subsets of $\D$, which
finishes the proof of (1.a).

\noindent{\it Proof of (1.b).} By Theorem \ref{LCimpliesEF}, there is an evolution family $(\varphi _{s,t})$ of order $d$ associated with
$(f_{t}).$ Let $G:\D\times [ 0,+\infty )\rightarrow \C$ be the Herglotz vector field whose positive trajectories
are $(\varphi _{s,t})$ (such a vector field exists by Theorem \ref{EF-bijection-VF}). Let $N_{1}\subset [
0,+\infty )$ be the set of zero measure given by \cite[Theorem 6.6]{BCM1} such that $\frac{\partial \varphi
_{0,u}}{\partial u}(z)=G(\varphi _{0,u}(z),u)$ for all $u\in (0,+\infty )\setminus N_{1}$ and all $z\in \D$.
Let $N_{2}\subset [0,+\infty )$ stand for the set of zero measure which has been denoted by $N$ in part (1.a) of this theorem.

Let $N:=N_{1}\cup N_{2}$. Differentiating with respect to $t$ the equality $ f_{t}(\varphi _{0,t}(z))=f_{0}(z)$,
for $z\in \D$ and $t\in (0,+\infty )\setminus N$ we obtain
\begin{align*}
0& =f_{t}^{\prime }(\varphi _{0,t}(z))\frac{\partial \varphi _{0,t}}{
\partial t}(z)+\frac{\partial f_{t}}{\partial t}(\varphi _{0,t}(z)) \\
& =f_{t}^{\prime }(\varphi _{0,t}(z)))G(\varphi _{0,t}(z),t)+\frac{\partial f_{t}}{\partial t}(\varphi
_{0,t}(z)).
\end{align*}
Therefore, $f_{t}^{\prime }(w)G(w,t)=-\frac{\partial f_{t}}{\partial t}(w)$ for all $w\in \varphi _{0,t}(\D)$.
Since the $\varphi _{0,t}$'s are univalent, the identity principle for holomorphic maps implies that this equality is valid for the whole unit disk~$\UD$.

\noindent{\it Proof of (2).} Fix a point $z$ in the unit disk. Then, up to a set of measure zero, we have
\begin{eqnarray*}
\frac{\partial }{\partial t}\left( f_{t}(\varphi _{s,t}(z))\right) &=&f_{t}^{\prime }(\varphi
_{s,t}(z))\frac{\partial \varphi _{s,t}}{\partial
t}(z)+\frac{\partial f_{t}}{\partial t}(\varphi _{s,t}(z)) \\
&=&f_{t}^{\prime }(\varphi _{s,t}(z))\frac{\partial \varphi _{s,t}}{\partial
t}(z)-G(\varphi _{s,t}(z),t)f_{s}^{\prime }(\varphi _{s,t}(z)) \\
&=&f_{t}^{\prime }(\varphi _{s,t}(z))\left[ \frac{\partial \varphi _{s,t}}{%
\partial t}(z)-G(\varphi _{s,t}(z),t)\right] =0.
\end{eqnarray*}
Therefore, $f_{t}(\varphi _{s,t}(z))$ does not depend on $t$. Hence,  $f_{t}(\varphi
_{s,t}(z))=f_s(\varphi_{s,s}(z))=f_{s}(z)$ and  the proof finishes just by applying Lemma \ref{equationimpliesLC}.
\end{proof}

\section{Remarks about semigroups}\label{semi}

A \textit{(one-parameter) semigroup of holomorphic functions} is a  continuous homomorphism
$\Phi:t\mapsto\Phi(t)=\phi_{t}$ from the additive semigroup of non-negative real numbers into the composition
semigroup of holomorphic self-maps of $\mathbb{D}$. Namely, $\Phi$ satisfies the following three conditions:

\begin{enumerate}
\item[S1.] $\phi_{0}$ is the identity in $\mathbb{D},$

\item[S2.] $\phi_{t+s}=\phi_{t}\circ\phi_{s},$ for all $t,s\geq0,$

\item[S3.] $\phi_{t}(z)$ tends to $z$ as $t$ tends to $0,$ uniformly on
compact subsets of $\mathbb{D}.$
\end{enumerate}

Let $(\phi_{t})$ be a semigroup of holomorphic self-maps of $\D$. Let $\varphi_{s,t}:=\phi_{t-s}$ for  $0\leq
s\leq t<+\infty.$ Then, by \cite[Example 3.4]{BCM1}, $(\varphi_{s,t})$ is an evolution family of order $\infty.$

Given a semigroup $\Phi=(\phi_{t})$, it is well-known (see \cite{Shoikhet}, \cite{Berkson-Porta}) that there
exists a \textit{unique} holomorphic function $G:\mathbb{D\rightarrow C}$ such that,
\[
\frac{\partial\phi_{t}(z)}{\partial t}=G\left(  \phi_{t}(z)\right)  =G\left( z\right)
\frac{\partial\phi_{t}(z)}{\partial z}\quad\text{for all } z\in\D\text{ and }t\geq0.
\]
The function $G$ is known as the {\it infinitesimal generator} of the semigroup and, obviously, $G$ (that clearly does not depend on $t$) is the Herglotz vector field associated with the evolution family $(\varphi_{s,t})$. Berkson and Porta~\cite{Berkson-Porta} proved that there exist $\tau\in\overline{\D}$ and a holomorphic function $p:\D\rightarrow \C$ with $\Re p(z)\geq0$
such that
\[
G(z)=(\tau-z)(1-\overline{\tau}z)p(z),\text{ \quad}z\in\mathbb{D},
\]
and moreover,  any function $G$ of this form is the infinitesimal generator of some semigroup.

In this very particular case when the  evolution family is generated by a semigroup, the point $\tau$ has a dynamical meaning. To explain this meaning, we have to recall some notions from iteration theory.

It can be  easily deduced from the Schwarz-Pick lemma that a non-identity self-map $\phi$ of the unit disc  can
have at most one fixed point in ${\mathbb{D}}$. If such a unique fixed point in ${\mathbb{D}}$ exists, it is
usually called the \textit{Denjoy-Wolff point\/}. The sequence of iterates $\{\phi_{n}\}$ of $\phi$ converges to
it uniformly on the compact subsets of $\D$ whenever $\phi $ is not a disc automorphism.

If $\phi$ has no fixed points in $\D$, the Denjoy-Wolff theorem
(see, e.\,g., \cite{Abate}) guarantees the existence of a unique
point $\tau$ on the unit circle $\partial{\mathbb{D}}$ which is the
\textit{attractive fixed point\/}, that is, the sequence of iterates
$\{\phi_{n}\}$ converges to $\tau$ uniformly on the compact subsets
of $\D$. Such a point $\tau$ is again called the
\textit{Denjoy-Wolff point\/} of $\phi$. When $\tau\in\partial\D$ is
the Denjoy-Wolff point of $\phi$,  the angular derivative
$\phi^{\prime}(\tau)$ is actually real-valued and, moreover,
$0<\phi^{\prime}(\tau)\leq1$ (see \cite{Pommerenke-II}). As it is
often done in the literature, we classify the holomorphic self-maps
of the disc into three categories according to their behavior near
the Denjoy-Wolff point:

\begin{itemize}
\item[(a)] \textit{elliptic\/}: the ones with a fixed point inside the unit disc~$\UD$;

\item[(b)] \textit{hyperbolic\/}: the ones with the Denjoy-Wolff point
$\tau\in\partial{\mathbb{D}}$ such that $\phi^{\prime}(\tau)<1$;

\item[(c)] \textit{parabolic\/}: the ones with the Denjoy-Wolff point $\tau
\in\partial{\mathbb{D}}$ such that $\phi^{\prime}(\tau)=1$.
\end{itemize}

Going back to semigroups, we have to say that the point $\tau$ that appears in the Berkson-Porta representation
formula for the infinitesimal generator of the semigroup $(\phi_t)$ is the Denjoy-Wolff point of all the
functions $\phi_t$. In particular, all the functions share the Denjoy-Wolff point. But something more can be
said. If there is $t_0>0$ such that the function $\phi_{t_0}$ is elliptic (resp. hyperbolic, parabolic) then
all the functions of the semigroup are elliptic (resp. hyperbolic, parabolic).

Besides the above classification of self-maps of the unit disk, there are two quite different types of parabolic
functions. To distinguish such functions, we have to recall the notion of hyperbolic step. Given a holomorphic
self-map $\phi$ of $\D$ and a point $z_0$ in $\D$, we define the \textit{forward orbit\/} of $z_0$ under $\phi$
as the sequence $z_{n}=\phi_{n}(z_{0})$. It is customary to say that $\phi$ is of \textit{zero hyperbolic
step\/} if for some point $z_{0}$ the orbit $z_{n}=\phi_{n}(z_{0})$ satisfies the condition
$\lim_{n\to\infty} \rho_\D(z_{n},z_{n+1})=0$. It is well-known that the word \textquotedblleft
some\textquotedblright\ here can be replaced by \textquotedblleft all\textquotedblright. In other words, the
definition does not depend on the choice of the initial point of the orbit (see, for example,
\cite{Contreras-Diaz-Pommerenke:zoo}).

Using the Schwarz-Pick Lemma, it is easy to see that the maps which are not of zero hyperbolic step are
precisely those holomorphic self-maps $\phi$ of $\D$ for which
\[
 \lim_{n\to\infty}\rho_\D(z_{n},z_{n+1})>0 \,,
\]
for some forward orbit $\{z_n\}_{n=1}^\infty$ of $\phi$, and hence for all such orbits. This is the reason why
they are called \textit{maps of positive hyperbolic step}. For a survey of these properties, the reader may
consult \cite{Contreras-Diaz-Pommerenke:zoo}.

It is easy to show that if $\phi$ is elliptic and is not an automorphism, then it is of zero hyperbolic step. If
$\phi $ is hyperbolic, then it is of positive hyperbolic step. For parabolic maps the situation is more
complicated: there are parabolic functions of zero hyperbolic step and of positive hyperbolic step. For example,
the following dichotomy holds for parabolic linear-fractional maps: every parabolic automorphism of $\D$ is of
positive hyperbolic step, while all non-automorphic linear-fractional parabolic self-maps of $\D$ are of zero
hyperbolic step. For semigroups of holomorphic functions we can state the following

\begin{lemma}
Let $(\phi_t)$ be a semigroup of parabolic functions in the unit disk. If there exists $t_0>0$ such that the
function $\phi_{t_0}$ is of zero hyperbolic step, then all the functions $\phi_t$, with $t>0$, of the semigroup are of zero hyperbolic step.
\end{lemma}
\begin{proof} In this proof we will use different well-known properties of the hyperbolic distance on
simply connected domains in the complex plane that can be seen in \cite{Shapiro}.

By \cite{Siskakis}, there exists a univalent function $h:\D\to\C$, with $h(0)=0$, such that $h\circ \phi_t =h+t$
for all $t>0$. Write $\Omega=h(\D)$ and denote by $\delta_\Omega (w)$ the Euclidean distance from $w\in \Omega$
to $\partial \Omega$. Since $\Omega +t\subseteq\Omega$ for all $t>0$, we can easily obtain that the function
$\delta_\Omega :[0,+\infty)\to\R$ is non-decreasing (we are considering here the restriction of $\delta_\Omega$
to the half-line $[0,+\infty)$). By hypothesis, the sequence $\rho_\D (\phi_{nt_0}(0),\phi_{(n+1)t_0}(0))$ goes
to zero. Moreover, by the Distance Lemma, we have that
\begin{eqnarray*}
  \rho_\D (\phi_{nt_0}(0),\phi_{(n+1)t_0}(0)) &=&  \rho_\Omega (h(0)+nt_0,h(0)+(n+1)t_0)
   =  \rho_\Omega (nt_0,(n+1)t_0) \\
   &\geq& \frac{1}{2} \log \left(1+\frac{|t_0|}{\min\{\delta_\Omega(nt_0),\delta_\Omega((n+1)t_0)\}}\right),
\end{eqnarray*}
where $\rho_\Omega$ denotes the hyperbolic distance on $\Omega$. Thus $\delta_\Omega((n+1)t_0)$ goes to $\infty$
and we conclude that $\lim_{t\to+\infty}\delta_\Omega(t)=\infty.$

Now fix $t>0$. Write $\Gamma _n=[nt,(n+1)t]$ and denote by $l_\Omega
(\Gamma_n)$ the hyperbolic length of $\Gamma _n$ in $\Omega$. We
have
$$
   \rho_\D (\phi_{nt}(0),\phi_{(n+1)t}(0)) = \rho_\Omega (nt,(n+1)t)
   \leq  l_\Omega (\Gamma_n)
   \leq  2\int_{\Gamma_n}\frac{|dw|}{\delta_\Omega (w)}
   \leq 2 \frac{1}{\delta_\Omega (nt)},
$$
where again we have used the monotonicity of $\delta_\Omega$ on $[0,+\infty)$. Since the sequence
$\delta_\Omega((n+1)t)$ goes to $\infty$, the above inequality implies that $\rho_\D
(\phi_{nt}(0),\phi_{(n+1)t}(0))$ tends to zero as $n$ goes to $\infty$. The arbitrariness of $t$ concludes the
proof.
\end{proof}

What can we say about the Loewner chains associated with the evolution families $(\varphi_{s,t})=(\phi_{t-s})$?

If the semigroup $(\phi_{t})$ is elliptic and its Denjoy-Wolff point is zero, then (see \cite{Siskakis}) there is a complex number
$c$ and a univalent function $h$ such that $\Re c\geq 0$, $h(0)=0$, $h'(0)=1$, and
\begin{equation}\label{Schr}
h\circ \phi_t=e^{-ct}h.
\end{equation}
The function $h$ is called the {K\oe nigs function} of the semigroup $(\phi_t)$.
>From equation~\eqref{Schr}, it is clear that the functions $f_t= e^{ct}h$ form a normalized
Loewner chain associated with the evolution family $(\varphi_{s,t})=(\phi_{t-s})$. If $\Re c>0$, then
$\cup_{t\geq 0} f_t(\D)=\C$ and, by Theorem~\ref{Gum_uniq}, this is the unique normalized Loewner chain associated with
$(\varphi_{s,t})$. In particular, this implies the uniqueness of the K\oe nigs function, a fact which is very
well-known. If $\Re c=0$, then  $h$ is the identity map.

Now suppose that the Denjoy-Wolff point of the semigroup is on the boundary of the unit disc. Without loss of
generality, we assume that such a point is 1. Then there is a univalent function $h$ such that $h(0)=0$ and
$h\circ \phi_t =h+t$ \cite{Siskakis}. As in the elliptic case, the function $h$ is referred to as the {K\oe nigs function} of the semigroup~$(\phi_t)$. Siskakis \cite{Siskakis-tesis} (see also \cite{Siskakis}) proved that the K\oe nigs function is unique in this case as well. As an easy application of our results we will reprove the uniqueness of the K\oe nigs function. Similarly to the elliptic case, we have that the functions $f_t=h-t$ form a Loewner chain associated with the evolution family
$(\varphi_{s,t})=(\phi_{t-s})$. Notice that this Loewner chain is not necessarily normalized. To proceed let us
distinguish the different type of semigroups.

If the functions $\phi_t $ are hyperbolic, then by \cite[Theorem 2.1]{Contreras-Diaz:pacific}, there is a horizontal
strip $\Omega$ such that the range of $h$ is included in $\Omega$ and $\cup_{t\geq 0} f_t(\D)=\Omega$. In this
case, there are much more Loewner chains associated with $(\varphi_{s,t})$ but there is no other Loewner chain
$(g_t)$ of the form $g_t=k-t$, where $k$, $k(0)=0$, is a univalent holomorphic function in~$\UD$. Indeed, if such another function $k$ does exist, then by Theorem \ref{Gum_uniq}, there is a univalent holomorphic function $a: \Omega \to \C$ such that $a (h(z)-t)=k(z)-t$ for all $t\geq0$ and for all $z\in\D$. Derivating with respect to $t$, we have $a' (h(z)-t)=1$ for all $t\geq0$ and for all $z\in\D$. That is $a(z)=z+c$ for some constant $c$. Therefore $h(z)-t+c=k(z)-t$ for all $t\geq0$ and for
all $z\in\D$. Since $h(0)=k(0)=0$, we deduce that $c=0$ and $h=k$.

Consider now the parabolic case. According to the above lemma we have to distinguish two subcases. On one hand, if
for some (or for any) $t_0>0$ the function $\phi_{t_0}$ is of zero hyperbolic step, then by \cite[Theorem
3.1 and Proposition 3.3]{Contreras-Diaz-Pommerenke:abel}, the range of $h$ is not included in any horizontal half-plane. In this case, we have that $\cup_{t\geq 0} f_t(\D)=\C$.
Therefore, up to normalization, this is a unique Loewner chain associated with $(\varphi_{s,t})$. On the other
hand, if one (and then all) of the mappings $\phi_t$ is of positive hyperbolic step, then the range of
$h$ is included in a horizontal half-plane $\Omega$. In fact, we can choose the half-plane such that
$\cup_{t\geq 0} f_t(\D)=\Omega$. By the same reason as in the hyperbolic case, there are much more Loewner chains associated with $(\varphi_{s,t})$ but there is no other Loewner chains $(g_t)$ of the form $g_t=k-t$, where $k$, $k(0)=0$, is a univalent holomorphic function in~$\UD$. That is, again, the K\oe nigs function of the semigroup is unique.

\end{document}